\documentclass[12pt]{article}
 
 \usepackage{amsmath, latexsym, amsfonts, amssymb, amsthm, amscd, enumerate, comment, stackrel, mathrsfs}
 \usepackage[usenames,dvipsnames,svgnames,table]{xcolor}
 \usepackage[left=3cm, right=2.5cm, top=2.5cm, bottom=2.5cm]{geometry}
\usepackage[affil-it]{authblk}
\usepackage{blindtext}
\usepackage{hyperref}

 \newtheorem{theorem}{Theorem}[section]
 \newtheorem{proposition}[theorem]{Proposition}
 \newtheorem{lemma}[theorem]{Lemma}

 \theoremstyle{definition}

 \title{Speed of extinction for CSBP in a subcritical L\'evy environment:  strongly and intermediate cases}
 \author{Natalia Cardona-Tob\'on and Juan Carlos Pardo}

\begin{document}
 	
 	\begin{center}
 		{\Large \bf Explosion rates for continuous state branching processes in a L\'evy environment}\\[5mm]

 		\vspace{0.7cm}
 		\textsc{Natalia Cardona-Tob\'on * \footnote{Institute for Mathematical Stochastics, Georg-August-Universit\"{a}t G\"{o}ttingen.  Goldschmidtstrasse 7, 37077 Göttingen, Germany,  {\tt natalia.cardonatobon@uni-goettingen.de}}}  and \textsc{Juan Carlos Pardo\footnote{Centro de Investigaci\'on en Matem\'aticas. Calle Jalisco s/n. C.P. 36240, Guanajuato, M\'exico, {\tt jcpardo@cimat.mx}}}
 		\vspace{0.5cm}

 		{\textit{*Corresponding author: natalia.cardonatobon@uni-goettingen.de}}

 	\end{center}
 	
 	\vspace{0.3cm}

 	\begin{abstract}
 		\noindent 
 		Here we study the long-term behaviour of the non-explosion probability  for continuous-state branching processes in a L\'evy environment when the branching mechanism is given by the negative of the Laplace exponent of a subordinator. In order to do so, we study the law of this family of processes in the infinite mean case and provide necessary and sufficient conditions for the process to be conservative, i.e. that the process does not explode in finite time a.s. In addition, we establish  precise rates for the non-explosion probabilities in the subcritical and critical regimes, first found by  Palau et al. \cite{palau2016asymptotic} in the case when the branching mechanism is  given by the negative of the Laplace exponent of a stable subordinator.
 		
 	\end{abstract}

 	\section{Introduction and main results}

Let us  consider $(\Omega^{(b)}, \mathcal{F}^{(b)}, (\mathcal{F}^{(b)}_t)_{t\geq 0}, \mathbb{P}^{(b)})$ a filtered probability space satisfying the usual hypothesis. A continuous-state branching process (CSBP for short) $Y=(Y_t, t\ge 0)$ defined on $(\Omega^{(b)}, \mathcal{F}^{(b)}, (\mathcal{F}^{(b)}_t)_{t\geq 0}, \mathbb{P}^{(b)})$ is a $[0,\infty]$-valued Markov process with c\`adl\`ag paths whose laws satisfy  the branching property, that is the law of $Y$  started from $x+y$ is the same as the law of the   sum of two independent copies of $Y$ but issued from $x$ and $y$, respectively. More precisely, for all $\lambda\ge 0$ and $x,y \ge 0 $
\begin{equation*}
	\mathbb{E}^{(b)}_{x+y}\left[e^{-\lambda Y_t}\right] = \mathbb{E}_x^{(b)}\left[e^{-\lambda Y_t}\right] \mathbb{E}_y^{(b)}\left[e^{-\lambda Y_t}\right].
\end{equation*}
Moreover, the law of $Y$ is completely characterised by the latter identity (see for instance Theorem 12.1 in \cite{kyprianou2014fluctuations}), i.e. for all $y>0$
\begin{equation*}
	\mathbb{E}^{(b)}_y\left[e^{-\lambda Y_t}\right] = e^{-yu_t(\lambda)},
\end{equation*}
where $t\mapsto u_t(\lambda)$ is a differentiable function  satisfying the following differential equation
\begin{equation}\label{eq:equationu}
	\frac{\partial }{\partial t} u_t(\lambda) + \psi (u_t(\lambda)) = 0 \quad \text{and} \quad u_0(\lambda) = \lambda,
\end{equation}
where the function $\psi$ is called the \textit{branching mechanism} of the CSBP $Y$. It is well-known that
the function $\psi$ is either the negative of the Laplace exponent of a subordinator or the Laplace exponent of a spectrally negative L\'evy process (see e.g. Theorem 12.1 in \cite{kyprianou2014fluctuations}). In this manuscript, we focus on the subordinator case, that is  we  assume 
\begin{equation}\label{eq_phi_subor}
	\psi(\lambda)= - \delta \lambda - \int_{(0,\infty)} (1-e^{-\lambda x}) \mu(\mathrm{d} x), \qquad \lambda \geq 0,
\end{equation}
where $\delta\ge 0$ and $\mu$ is a  measure supported in $(0,\infty)$ satisfying
\begin{equation}\label{muint}
	\int_{(0,\infty)}(1\wedge x)\mu(\mathrm{d} x)<\infty. 
\end{equation}
It is important to note  that the function $-\psi$ is also known in the literature as a \textit{Bernstein function}, see for instance the monograph of Schilling et al. \cite{ScSVon} for further details about this class of functions.

In this manuscript,  we are interested in a particular extension of CSBPs by considering a  random  external force which affects the dynamics of the previous model. More precisely, we consider continuous-state branching processes in a random environment. Roughly speaking, a process in this class is a time-inhomogeneous Markov process taking values in $[0,\infty]$ with $0$ and 
$\infty$ as absorbing states. Furthermore, such processes satisfies a quenched branching property; that is  conditionally on the environment,  the process started from $x+y$ is distributed as the  sum of two independent copies of the same process but issued from $x$ and $y$, respectively.

Actually the subclass of CSBPs  in random environment that we are interested in pertains to instances where the external stochastic perturbation is driven by an independent L\'evy process. This family of processes is known as 
\emph{continuous-state branching processes in L\'evy environments} (or CBLEs for short) and 
its construction have been given  by  He et al.
\cite{he2018continuous} and by  Palau and Pardo \cite{palau2018branching}, independently, as the unique non-negative strong solution of a stochastic differential equation which will  be specified below.   

We now construct the demographic or branching terms of the model.
Let  $N^{(b)}(\mathrm{d} s , \mathrm{d} z, \mathrm{d} u)$ be a  $(\mathcal{F}_t^{(b)})_{t\geq 0}$-adapted Poisson random measure on $\mathbb{R}^3_+$ with intensity $\mathrm{d} s \mu(\mathrm{d} z)\mathrm{d} u$, where $\mu $ is defined as above and satisfies the integral condition in \eqref{muint}.  On the other hand, for the environmental term, we  consider another filtered probability space $(\Omega^{(e)}, \mathcal{F}^{(e)},(\mathcal{F}^{(e)}_t)_{t\geq 0}, \mathbb{P}^{(e)})$   satisfying the usual hypotheses.   Let us consider $\sigma \geq 0$ and $\alpha$  real constants;  and $\pi$  a measure concentrated on $\mathbb{R}\setminus\{0\}$ such that $$\int_{\mathbb{R}} (1\land z^2)\pi(\mathrm{d} z)<\infty.$$ Suppose that \ $(B_t^{(e)}, t\geq 0)$ \ is a $(\mathcal{F}_t^{(e)})_{t\geq0}$ - adapted standard Brownian motion, $N^{(e)}(\mathrm{d} s, \mathrm{d} z)$ is a $(\mathcal{F}_t^{(e)})_{t\geq 0}$ - Poisson random measure on $\mathbb{R}_+ \times \mathbb{R}$ with intensity $\mathrm{d} s \pi(\mathrm{d} z)$, and $\widetilde{N}^{(e)}(\mathrm{d} s, \mathrm{d} z)$ its compensated version. We denote by $S=(S_t, t\geq 0)$ a L\'evy process, that is  a process with  stationary and independent increments and  c\`adl\`ag paths, with  the following L\'evy-It\^o decomposition
\begin{equation*}\label{eq_ambLevy}
	S_t = \alpha t + \sigma B_t^{(e)} + \int_{0}^{t} \int_{(-1,1)} (e^z - 1) \widetilde{N}^{(e)}(\mathrm{d} s, \mathrm{d} z) + \int_{0}^{t} \int_{(-1,1)^c} (e^z - 1) N^{(e)}(\mathrm{d} s, \mathrm{d} z).
\end{equation*}
Note that  $S$ has no jumps smaller than or equals to -1 since the size of the jumps are given by the map $z\mapsto e^z-1$.

In our setting, the population size has no impact on the evolution of the environment and we are considering independent processes for the demography and the environment. More precisely, we work now on the space $(\Omega, \mathcal{F}, (\mathcal{F}_t)_{t\geq 0}, \mathbb{P})$ which is the direct product of the two probability spaces defined above, that is to say, $\Omega := \Omega^{(e)} \times \Omega^{(b)}, \mathcal{F}:= \mathcal{F}^{(e)}\otimes \mathcal{F}^{(b)},  \mathcal{F}_t:=  \mathcal{F}^{(e)}_t \otimes  \mathcal{F}^{(b)}_t$ for $t\geq0$, and $ \mathbb{P}:=\mathbb{P}^{(e)} \otimes \mathbb{P}^{(b)} $. Therefore,
\textit{the continuous-state branching process $Z=(Z_t, t\geq 0)$ in a L\'evy environment $S$ }  is defined on  $(\Omega, \mathcal{F}, (\mathcal{F}_t)_{t\geq 0}, \mathbb{P})$ as the unique non-negative strong solution of the following SDE
\begin{equation}\label{CBILRE}
	Z_t = Z_0 +\delta \int_{0}^{t} Z_s \mathrm{d} s + \int_{0}^{t} \int_{[0,\infty)} \int_{0}^{Z_{s-}}z N^{(b)} (\mathrm{d} s, \mathrm{d} z, \mathrm{d} u)   + \int_{0}^{t} Z_{s-} \mathrm{d} S_s.
\end{equation}
According to Theorem 1 in \cite{palau2018branching}, the SDE \eqref{CBILRE} has pathwise uniqueness and strong solution  up to explosion and by convention here it is identically equal to $\infty$ after the explosion time.  Further, the process $Z$ satisfies the strong Markov property and the \textit{quenched branching property}.  For further details, we refer to He et al. \cite{he2018continuous} or  Palau and Pardo \cite{palau2018branching} in more general settings.


 According to He et al. \cite{he2018continuous} and Palau and Pardo \cite{palau2018branching}, the long-term behaviour of the process $Z$ is deeply related to the behaviour and fluctuations of the L\'evy process $\xi = ( \xi_t , t \ge  0)$, defined as follows
\begin{equation}\label{eq_envir2}
	\xi_t = \overline{\alpha} t + \sigma B_t^{(e)} + \int_{0}^{t} \int_{(-1,1)} z \widetilde{N}^{(e)}(\mathrm{d} s, \mathrm{d} z) + \int_{0}^{t} \int_{(-1,1)^c}z N^{(e)}(\mathrm{d} s, \mathrm{d} z),
\end{equation}
where
\begin{equation*}
	\overline{\alpha} := \alpha+\delta -\frac{\sigma^2}{2} - \int_{(-1,1)} (e^z -1 -z) \pi(\mathrm{d} z).
\end{equation*}
Note that, both processes $(S_t, t\geq 0)$ and $(\xi_t, t\geq 0)$ generate the same filtration. Actually, the process $\xi$ is obtained from $S$ by changing  the drift term and the jump sizes as follows
\[
\alpha \mapsto \overline{\alpha}\qquad \textrm{and} \qquad \Delta S_t\mapsto \ln \left(\Delta S_t +1\right), 
\]
where $\Delta S_t=S_t-S_{t-}$, for $t\ge 0$. Furthermore, we point out that the drift term $\overline{\alpha}$ involves both branching and environment parameters. The relevance of this process stems from the fact that, in the finite mean case, the process  $(Z_te^{-\xi_t}, t\geq 0)$ is a quenched martingale which allows to study the long-term behaviour of the process $Z$ (see Proposition 1.1 in \cite{bansaye2021extinction}). Moreover, the law of the process $(Z_te^{-\xi_t}, t\geq 0)$ can be  characterised via a backward differential equation which   is the analogue to \eqref{eq:equationu} when the environment is fixed.  In the infinite mean case, it is no longer true that  $(Z_te^{-\xi_t}, t\geq 0)$  is a quenched martingale, however, in this paper we show that the law of $(Z_te^{-\xi_t}, t\geq 0)$ can be also characterised via a backward differential equation.

Recall that  we say that $Z$ is a \textit{conservative} process  if
\begin{equation}\label{eq_conservative}
	\mathbb{P}_z(Z_t<\infty)=1,\quad \quad \text{for all}\quad t>0,
\end{equation}
where $\mathbb{P}_z$ denotes the law of $Z$ starting at  $z\geq 0$. Thus when we think about the non-explosion event, the first question that might arise is: under which conditions the process $Z$ is  conservative?   When the environment is  fixed, Grey in \cite{grey1974asymptotic} provided a necessary and sufficient condition for the process to be conservative which depends on the integrability at $0$ of  the associated branching mechanism $\psi$.   
More precisely,  a continuous-state branching process  is conservative if and only if 
\begin{equation*}
	\int_{0+} \frac{1}{|\psi(\lambda)|} \mathrm{d} \lambda = \infty.
\end{equation*}
Observe that, a necessary condition is that $\psi(0)=0$ and a sufficient condition is that $\psi(0)=0$ and $|\psi'(0+)|<\infty$ (see for instance Theorem 12.3 in  \cite{kyprianou2014fluctuations}).  In contrast, in the particular case when the environment is driven by a  L\'evy process, Bansaye et al. \cite{bansaye2021extinction} furnish  a necessary condition under which the process $Z$ is conservative. They proved that if the branching mechanism satisfies $|\psi'(0+)|<\infty$, then the associated CBLE is conservative (see Lemma 7 in \cite{bansaye2021extinction}). It is worth noting that these results remain valid when the branching mechanism is given by the Laplace exponent of a spectrally negative L\'evy process.  Nevertheless, our focus here is on the case when the branching mechanism is given by \eqref{eq_phi_subor} and satisfies $\psi'(0+)=-\infty$. So our first aim is twofold, first we characterise the law of the process  $(Z_te^{-\xi_t}, t\geq 0)$  in the infinite mean case, which up to our knowledge is unknown, and then provide necessary and sufficient conditions for conservativeness when the branching mechanism is given as in \eqref{eq_phi_subor}  and satisfies $\psi'(0+)=-\infty$. 


Our second aim deals with the asymptotic behaviour of the non-explosion probability which,  up to our knowledge, has only been  studied for the case where the associated branching mechanism is   given by the negative of the Laplace exponent of a stable subordinator (see Proposition 2.1 in  \cite{palau2016asymptotic}), that is  
\begin{equation}\label{est}
\psi(\lambda)= C\lambda^{1+\beta},\quad \text{for all}\quad \lambda\ge0,
\end{equation}
where  $ \beta \in (-1,0)$ and $C$ is a negative constant. According to \cite{palau2018branching},  the non-explosion probability for a  CBLE  with branching mechanism as in \eqref{est} (stable CBLE for short) is given by
\begin{equation}\label{eq_noexplosion}
	\mathbb{P}_z(Z_t<\infty ) = \mathbb{E}\left[\exp\left\{-z\big(\beta C\texttt{I}_{0,t}(\beta \xi)\big)^{-1/\beta}\right\}\right], \qquad z\geq 0,
\end{equation}
where $\texttt{I}_{0,t}(\beta \xi)$  denotes the exponential functional of  the L\'evy process $\beta \xi$,  i.e.
\begin{equation}\label{eq_expfuncLevy}
	\texttt{I}_{0,t}(\beta \xi): = \int_{0}^{t} e^{-\beta \xi_s} \mathrm{d} s.
\end{equation}
Observe from  \eqref{eq_noexplosion},  that the process $Z$ explodes  with positive probability. Hence our second aim  is study  the rates of the non-explosion probability of CBLEs in a more general setting rather than the stable case. 

According to Palau et al. \cite{palau2016asymptotic} there are three different regimes for the asymptotic behaviour of the non-explosion probability that depends only  on the mean of the underlying L\'evy process $\xi$. They called these regimes: \textit{subcritical, critical} and  \textit{supercritical explosive} depending on whether the mean of  $\xi$ is negative, zero or positive  (see Proposition 2.1 in  \cite{palau2016asymptotic}).
%
Up to our knowledge, this is the only known result in the literature about explosion rates for CBLEs. 

  In particular, we study the speed of the non-explosion probability of the process $Z$, when the branching mechanism $\psi$ is given as in \eqref{eq_phi_subor},
 in the {\it critical} and {\it subcritical explosive} regimes. In particular, we show that  in the subcritical explosive regime, i.e. when the auxiliary L\'evy process $\xi$ drifts to $-\infty$, and under an integrability condition, the limit of the non-explosion probability is positive. In the critical regime, i.e. when $\xi$ oscillates, and in particular when $\xi$ satisfies   the so-called Spitzer's condition plus an integrability condition, the  non-explosion probability decays as a regularly varying function at $\infty$. Our arguments  use strongly fluctuation theory and  the asymptotic behaviour of exponential functionals of L\'evy processes. 

The supercritical explosive regime, i.e. when $\xi$ drifts to $\infty$, remains unknown, except in the stable case. We may expect, similarly as in the stable case,  that under  some positive exponential moments of $\xi$, the   non-explosion probability decays exponentially as time increases. But it seems that this case requires other techniques than the ones developed in this article.


As we said before, the long-term behaviour of the non-explosion probability of  $Z$ is deeply related to the behaviour and fluctuations of  $\xi$. Therefore, a few knowledge on fluctuation theory of L\'evy process are required in order to state our main results. 
\subsection{Preliminaries on L\'evy processes}\label{sec_defandprop}
For simplicity, we denote  by $\mathbb{P}^{(e)}_x$  the law of the process $\xi$ 
starting from $x\in \mathbb{R}$, 
and when $x=0$, we use the notation $\mathbb{P}^{(e)}$ for $\mathbb{P}^{(e)}_0$ (resp. $\mathbb{E}^{(e)}$ for $\mathbb{E}_0^{(e)}$).   Let us denote by $\widehat{\xi}=-\xi$ the dual process. 
For every $x\in \mathbb{R}$,  let $\widehat{\mathbb{P}}_x^{(e)}$ be the law of $x+\xi$ under $\widehat{\mathbb{P}}^{(e)}$, that is the law of $\widehat{\xi}$ under $\mathbb{P}_{-x}^{(e)}$. In the sequel, we assume that $\xi$ is not a compound Poisson process since it is possible that in this case the process visits the same maxima or minima at distinct times which can make our analysis more involved.

Let us introduce the running infimum  and supremum  of $\xi$, by
$\underline{\xi}=(\underline{\xi}_t, t\geq 0)$ and $\overline{\xi}= (\overline{\xi}_t, t\geq 0)$, with
\begin{equation}\label{eq:runsupinf}
	\underline{\xi}_t = \inf_{0\leq s\leq t} \xi_s \qquad \textrm{ and } \qquad \overline{\xi}_t = \sup_{0 \leq s \leq t} \xi_s, \qquad t \geq 0.
\end{equation}
It is well-known that the  reflected process $\xi-\underline{\xi}$ (resp. $\overline{\xi}-\xi$)  is a Markov process with respect to the  filtration $(\mathcal{F}^{(e)}_t)_{t\geq 0}$, see for instance
Proposition VI.1 in \cite{bertoin1996levy}.  We 
denote by $L=(L_t, t \geq 0 )$ and $\widehat{L}=(\widehat{L}_t, t \geq 0 )$  the local times of $\overline{\xi}-\xi$ and $\xi-\underline{\xi}$ at $0$, respectively,  in the sense of Chapter IV in \cite{bertoin1996levy}.
Next, define  
\begin{equation}\label{defwidehatH}
	H_t=\overline{\xi}_{L_t^{-1}} \qquad \text{and}\qquad \widehat{H}_t=-\underline{\xi}_{\widehat{L}_t^{-1}} \qquad t\ge 0,
\end{equation}
where $L^{-1}$ and $\widehat{L}^{-1}$ are the right continuous inverse of the local times $L$ and $\widehat{L}$, respectively.  The range of the inverse local times, $L^{-1}$ (resp. $\widehat{L}^{-1}$), corresponds to the set of times at which new maxima (resp. new minima) occur. Hence, the range of the process $H$ (resp. $\widehat{H}$) corresponds to the set of new maxima (resp. new minima). The pairs $(L^{-1}, H)$ and  $(\widehat{L}^{-1}, \widehat{H})$ are bivariate subordinators known as the ascending and descending ladder processes, respectively.  The  Laplace transform of $(L^{-1}, H)$ is such that for $ \theta,\lambda \geq 0$, 
\begin{equation}\label{eq_kappa}
	\mathbb{E}^{(e)}\left[\exp\left\{-\theta L^{-1}_t-\lambda H_t \right\}\right]=\exp\left\{-t \kappa(\theta, \lambda)\right\},\qquad  t\ge 0, 
\end{equation}
where $\kappa(\cdot,\cdot)$ denotes its bivariate Laplace exponent (resp. $\widehat{\kappa}(\cdot, \cdot)$ for the  descending ladder process). Similarly to the  absorption rates studied in \cite{bansaye2021extinction, cardona2021speed, cardona2023speed}, the asymptotic analysis  of the event  of explosion and the role of the initial condition involve the renewal functions $U$ and $\widehat{U}$, associated to the supremum and infimum respectively,  which are defined, as follows
\begin{equation}\label{eq_renewalfns}
	U(x) := \mathbb{E}^{(e)}\left[\int_{[0,\infty)} \mathbf{1}_{\left\{H_t\leq x\right\}} \mathrm{d} t\right] \quad 
	\textrm{and}\quad
	\widehat{U}(x) := \mathbb{E}^{(e)}\left[\int_{[0,\infty)} \mathbf{1}_{\left\{\widehat{H}_t\leq x\right\}} \mathrm{d} t\right], \qquad x>0.
\end{equation}
The renewal functions $U$ and $\widehat{U}$ are  finite, subadditive, continuous and increasing. Moreover, they are identically 0 on $(-\infty, 0]$,  strictly positive on $(0,\infty)$  and satisfy 
 	\begin{equation}
 		\label{grandO}
 		U(x)\leq C_1 x \qquad\textrm{and}\qquad \widehat{U}(x)\leq C_2 x \quad \text{  for any } \quad x\geq 0,
 	\end{equation}
 	where $C_1, C_2$ are finite constants  (see for instance  Lemma 6.4 and Section 8.2 in the monograph of Doney 
 	\cite{doney2007fluctuation}). Moreover $U(0)=0$ if $0$ is regular upwards  and $U(0)=1$ otherwise, similarly $\widehat{U}(0)=0$ if $0$ is regular upwards  and $\widehat{U}(0)=1$ otherwise.
	
Roughly speaking, the renewal function $U(x)$ (resp.  $\widehat{U}(x)$)  ``measures'' the amount of time that the ascending (resp. descending) ladder height process spends on the interval $[0,x]$ and in particular induces a measure on $[0,\infty)$ which is known as the renewal measure.  The latter implies
 \begin{equation}\label{bivLap}
 	\int_{[0,\infty)} e^{-\theta x} U( \mathrm{d}x) = \frac{1}{\kappa(0,\theta)},  \qquad  \theta>0.
\end{equation}
A similar expression holds true for the Laplace transform of the measure $\widehat{U}(\mathrm{d}x)$ in terms of $\widehat{\kappa}$. For a more in-depth account of fluctuation theory, we refer the reader to the monographs of Bertoin \cite{bertoin1996levy}, Doney \cite{doney2007fluctuation} and Kyprianou \cite{kyprianou2014fluctuations}.

\subsection{Main results}\label{sec_explosionmainresult}
Our first main result determines the law of the reweighted process $(Z_te^{-\xi_t}, t\geq 0)$ via a backward  differential equation in terms of $\psi_0(\lambda):=\psi(\lambda)+\lambda\delta$ and the process $\xi$. The function $\psi_0$ plays now the role of the  branching mechanism for the  CBLE $Z$ since the demographic term $\delta$ can be added to the environment $S$ without changing the structure of $S$ and the definition of $Z$ in \eqref{CBILRE}. In fact,  we observe that the resulting process $(\delta t+S_t, t\ge 0)$ is still a L\'evy process. To avoid any confusion we call $\psi_0$ as the {\it pure branching mechanism} of the CBLE $Z$.

We denote by $ \mathbb{P}_{(z,x)}$ the law of the process $(Z, \xi)$ starting at $(z,x)$.

 	\begin{theorem}\label{teo_existencia}
 		For every $z > 0$, $x\in \mathbb{R}$, $\lambda\geq 0$ and $t\geq s\geq 0$, we have
 	\[
 			\mathbb{E}_{(z,x)}\Big[\exp\{-\lambda Z_te^{-\xi_t}\}\  \Big|\Big. \  \xi, \mathcal{F}^{(b)}_s\Big] = \exp\big\{- Z_se^{-\xi_s}v_t(s,\lambda,\xi)\big\},
 			\]
 		where for any $\lambda, t \geq 0$, the function $(v_t(s,\lambda, \xi), s\in [0,t])$ is an a.s. solution of the backward differential equation 
 		\begin{equation}\label{eq_BDEv}
 			\frac{\partial }{\partial s} v_t(s,\lambda,\xi) =e^{\xi_s} \psi_0\big(v_t(s,\lambda ,\xi)e^{-\xi_s}\big), \quad  \mbox{a.e.}\quad   s\in[0,t]
 		\end{equation}
 		and with terminal condition $v_t(t,\lambda, \xi)=\lambda$.  In particular, for every $z, \lambda, t > 0$ and $x\in \mathbb{R}$, we have
 		\begin{equation}\label{eq_quenchedlaw}
 			\mathbb{E}_{(z,x)}\Big[\exp\big\{-\lambda Z_t e^{-\xi_t}\big\} \Big|\Big. \  \xi\Big] = \exp\left\{-z v_t(0,\lambda e^{-x} ,\xi-x)\right\}.
 		\end{equation}
 	\end{theorem}
 	
 	Our second main result furnishes a necessary and sufficient condition for a CBLE to be conservative, in our setting. The result is an extension of the original characterisation given by Grey \cite{grey1974asymptotic} in the classical case for continuous-state branching process with constant environment. 
 	
 	\begin{proposition}\label{lem_condsufnec}
 		 A continuous-state branching process in a L\'evy environment with pure branching mechanism
 		$\psi_0$ is conservative if and only if  
 		\begin{equation}\label{eq_Grey}
 			\int_{0+} \frac{1}{|\psi_0(\lambda)|} \mathrm{d} \lambda = \infty.
 		\end{equation}
 	\end{proposition}
	 
 In what follows, we assume that  the pure branching mechanism $\psi_0$ satisfies
	\begin{equation}\label{nogrey}
	\int_{0+} \frac{1}{|\psi_0(\lambda)|} \mathrm{d} \lambda< \infty,
	\end{equation}
	or in other words that $Z$ may explode in finite time with positive probability.
	Under such assumption, we are interested in the asymptotic behaviour of the non-explosion probability in the  following  regimes: \textit{subcritical} and  \textit{critical explosive}.

	First we focus on the \textit{subcritical explosive} regime, i.e. when the L\'evy process drifts to $-\infty$. We recall that $\overline{\alpha}$ and $\pi$ are the drift term and the L\'evy measure of  $\xi$, respectively. 
 	We introduce the following real function
 	$$ A_{\xi}(x):= -\overline{\alpha} + \bar{\pi}^{(-)}(-1) + \int_{-x}^{-1}\bar{\pi}^{(-)}(y)\mathrm{d} y, \qquad \text{for}\quad  x>0,$$
 	where $\bar{\pi}^{(-)}(-x)=\pi(-\infty, -x)$. We also introduce the function
 	\begin{equation}\label{eq_phihat}
 		\Phi_\lambda(u):= \int_{0}^{\infty} \exp\{- \lambda e^{u} y\}  \bar{\mu}(y)\mathrm{d} y, \qquad \textrm{for}\quad \lambda > 0,
 	\end{equation}
	where $\bar{\mu}(x):= \mu(x,\infty)$.
 	Further, let us denote by $\texttt{E}_1$ the exponential integral, i.e.
 	\begin{equation}\label{eq_E1def}
 		\texttt{E}_1(w)= \int_{1}^{\infty} \frac{e^{-wy}}{y}\mathrm{d} y, \qquad w>0.
 	\end{equation}
 
 	We can then formulate the following theorem.
 	\begin{theorem}[Subcritical explosive regime]\label{teo_explosubcritica}
 		Suppose that \eqref{nogrey} holds, $\xi$ drifts to $-\infty$, $\mathbb{P}^{(e)}$ -a.s., and that there exists $\lambda>0$ such that
	\begin{equation}\label{eq_Axi}
 			\int_{(a,\infty)} \frac{y}{A_{\xi}(y)}|\mathrm{d}\Phi_\lambda(y)|<\infty, \qquad \text{for some}\quad  a>0.
 		\end{equation}	
 		Then, for any $z>0$, 
 		\begin{equation*}
 			\lim\limits_{t\to \infty}  \mathbb{P}_{z}(Z_t < \infty) >0.
 		\end{equation*}
 		In particular, if ~$\mathbb{E}^{(e)}\big[\xi_1\big]\in (-\infty, 0)$, then the integral condition \eqref{eq_Axi} is equivalent to 
 		\begin{equation}\label{eq_E1}
 			\int_{0}^{\infty} \normalfont{\texttt{E}_1}(\lambda y) \bar{\mu}(y)\mathrm{d} y<\infty.
 		\end{equation}
 	\end{theorem} Note that, for $y> 0$, the following inequality for the exponential integral holds
 	$$ \normalfont{\texttt{E}_1}(\lambda y) \leq  e^{-\lambda y}\log\left(1+\frac{1}{\lambda y}\right).$$
 Therefore, in the case  $\mathbb{E}^{(e)}\big[\xi_1\big]\in(-\infty, 0)$,  a simpler condition than \eqref{eq_E1} is the following 
	$$ \int_{0}^{\infty} e^{-\lambda y} \log\left(1+\frac{1}{\lambda y}\right)\bar{\mu}(y)\mathrm{d} y<\infty.$$
We also observe that since the mapping $\lambda\mapsto \Phi_\lambda(u)$ is decreasing, we have that if \eqref{eq_Axi} holds for some $\lambda>0$ then it holds for any $\tilde{\lambda}>\lambda$. Moreover, note that
$\log(1 + 1/y\lambda) = \log(1/y\lambda) + \log(y\lambda + 1)$ and since the problem is at zero only $\log(1/y\lambda)$ matters for the finiteness of the above integral.

 	Our next main result deals with  the \textit{critical explosive} regime.  More precisely,  we assume that $\xi$ satisfies the so-called \textit{Spitzer's condition}, i.e. 
 	\begin{equation}\label{eq_spitzer}\tag{\bf{A}}
 		\frac{1}{t}\int_{0}^{t} \mathbb{P}^{(e)}(\xi_s\geq 0) \mathrm{d} s \longrightarrow \rho \in (0,1), \quad \quad  \text{as}\quad \quad  t \to \infty.
 	\end{equation}
 	 Bertoin and Doney  in \cite{bertoin1997spitzer} showed that the later condition is  equivalent to $\mathbb{P}^{(e)} (\xi_s\geq 0) \to \rho$ as $s\to \infty$.
 	Recall from \eqref{eq_kappa} that $\kappa(\theta, \lambda)$ is the  Laplace exponent of the ascending ladder process $(L^{-1},H)$. Now, according to Theorem IV.12 in  \cite{bertoin1996levy},  Spitzer's condition \eqref{eq_spitzer} is equivalent to $\theta \mapsto\kappa( \theta , 0)$  being regularly varying at zero with index $\rho \in (0,1)$, that is 
 	\begin{equation*}
 		\lim\limits_{t\downarrow 0} \frac{\kappa(c t, 0)}{\kappa(t, 0)} =c^\rho, \quad \quad \text{for all}\quad \quad   c>0.
 	\end{equation*} 
 	In addition, the function $\theta \mapsto\kappa( \theta , 0)$ may always be written in the form
 	\begin{equation}\label{eq_k}
 		\kappa(t,0) =  t^\rho\, \widetilde{\ell}(t), 
 	\end{equation}
 	where $\widetilde{\ell}$ is a slowly varying function at $0^+$, i.e  for all positive constant $c$, it holds
 	$$ \lim\limits_{x\downarrow 0}\frac{\widetilde{\ell}(c x)}{\widetilde{\ell}(x)}=1.$$ 
 	 	According to Theorems VI.14 and VI.18 in \cite{bertoin1996levy} or  Remark 3.5 in \cite{kwasnicki2013suprema}, Spitzer's condition determines the asymptotic behaviour of the probability that the L\'evy process $\xi$ remains negative. In other words, under Spitzer's condition \eqref{eq_spitzer} and for $x<0$,  we have   \begin{equation}\label{eq_lim_M}
 		\lim\limits_{t\to \infty} \frac{\sqrt{\pi}}{\kappa(1/t, 0)} \mathbb{P}_{x}^{(e)}\big(\overline{\xi}_t<0\big) =\lim\limits_{t\to \infty} \frac{\sqrt{\pi}}{\kappa(1/t, 0)} \widehat{\mathbb{P}}_{-x}^{(e)}\big(\underline{\xi}_t> 0\big)= U(-x),
 	\end{equation}
 	where we recall that $U(\cdot)$ is the renewal function for the ascending ladder-height, defined in \eqref{eq_renewalfns}.
 	
 	
	In order to control the effect of the environment on the event of non-explosion we need other  assumptions. 
	The following integrability condition is needed to guarantee the non-explosion of the process in unfavourable environments. Let us assume
 	 	\begin{equation}\label{eq_Hyp_SubH}\tag{\bf{B}}
 		 		\int_{0^+} zU(-\ln(z))\widehat{U}(-\ln(z)) \mu(\mathrm{d} z)< \infty.
 			\end{equation} 
 		In particular, from \eqref{grandO}, we observe that the previous condition is satisfied if
 		\[
		\int_{0^+} z\ln^2(z) \mu(\mathrm{d} z)< \infty.
		\]

 	On the other hand, we shall assume that the {\it pure branching process} of the CBLE $Z$ is lower bounded by a stable branching mechanism whose associated CBLE  explodes with positive probability.  More precisely, we assume that 
 	\begin{equation}\label{eq_Hyp}\tag{\bf{C}}
 		\mbox{there exists} \  \beta \in (-1,0) \ \mbox{and}\ C<0 \mbox{ such that } \ 
 		\psi_0(\lambda) \geq C\lambda^{1+\beta} \quad \mbox{ for all } \lambda\geq 0.
 	\end{equation} 
 	The above condition  is necessary to deal with the functional $v_t(s,\lambda,\xi)$ and to obtain an upper bound for the speed of non-explosion when the sample paths of the L\'evy process have a high running  supremum (see Proposition \ref{prop_cota0} below for details). 
	
	Roughly speaking, our aim is to show, under the above conditions, that the probability of non-explosion varies regularly  at $\infty$ with index $\rho$.  
 	\begin{theorem}[Critical-explosion regime]\label{teo_explocritica} 	 		Suppose that \eqref{nogrey} holds and that  conditions, \eqref{eq_spitzer}, \eqref{eq_Hyp_SubH} and \eqref{eq_Hyp} are also fulfilled. Then, for any $z>0$, there exists $0< \mathfrak{C}(z) <\infty $ such that
 		\begin{equation*}
 			\lim\limits_{t\to \infty} \frac{1}{\kappa(1/t,0)} \mathbb{P}_{z}(Z_t < \infty) = \mathfrak{C}(z).
 		\end{equation*}
 	\end{theorem}
 The previous result provides evidence that the asymptotic behaviour  of the non-explosion probability is deeply related to the fluctuations of the L\'evy environment $\xi$.   
	
 	

\hspace{1cm}
 
The remainder of this paper is devoted to the proofs of the main results. In Section \ref{sec_explosionconserv}, we present the proofs of Theorem \ref{teo_existencia} and Proposition \ref{lem_condsufnec}. In Section \ref{sub_explosionsub}, the proof of Theorem \ref{teo_explosubcritica} is given. In Section \ref{sec_explosionconditioned}, we introduce  continuous-state branching processes in a conditioned L\'evy environment. This conditioned version is required to study the long-term behaviour of the non-explosion probability in the critical regime. Section  \ref{sec_explosiontheoremrates} is devoted to the long-term behaviour results for the critical regime.

 		\section{Proofs of Theorem \ref{teo_existencia} and Proposition \ref{lem_condsufnec}}\label{sec_explosionconserv}
 		
 		We first deal with the proof of Theorem \ref{teo_existencia} which relies on the extension of the classical Carath\'eodory's theorem for ordinary differential equations, that we state here for completeness. For its proof  the reader is referred to Theorems 1.1,  2.1 and 2.3 in Person \cite{persson1975generalization}.
 		
 		\begin{theorem}[Extended Carath\'eodory's  existence theorem]\label{teo_caratheo}
 			Let $I=[-b,b]$  with $b>0$.  Assume that the function $f: I\times \mathbb{R} \to \mathbb{R}$ satisfies the following  conditions: 
 			\begin{enumerate}
 				\item[i)] the mapping $s\mapsto f(s,\theta)$ is measurable for each fixed $\theta \in \mathbb{R}$,
 				\item[ii)] the mapping $\theta \mapsto f(s,\theta)$ is continuous for each fixed $s \in I$,
 				\item[iii)] there exists a Lebesgue-integrable function $m$ on the interval $I$ such that $$|f(s,\theta)| \leq m(s) \big( 1+|\theta|\big) , \quad (s,\theta)\in I\times \mathbb{R}.$$ 
 			\end{enumerate}
 			Then there exists an absolutely continuous function $u(x)$ such that 
 			\begin{equation}\label{eq_edo}
 				u(x) =\int_{0}^{x} f(y,u(y))\mathrm{d} y, \quad \quad  x\in I. 
 			\end{equation}
 		\end{theorem}

 	On the other hand, by integration by parts, we note that  the function $|\psi_0|=-\psi_0$ can also be rewritten as follows 	\begin{equation}\label{eq_phitail}
 		|\psi_0(\lambda)|= \lambda\int_{0}^\infty e^{-\lambda x} \bar{\mu}(x)\mathrm{d} x,
 	\end{equation}
 	where we recall that $\bar{\mu}(x)= \mu(x,\infty)$. The previous expression will be useful for what follows.

 		\begin{proof} [Proof of Theorem \ref{teo_existencia}] 
The first part of the proof follows from similar arguments as those used in \cite{bansaye2013extinction} and \cite{palau2018branching} in the case of finite mean (i.e., when $|\psi_0^\prime(0+)|<\infty$) whenever  there is an a.s.  solution of the backward differential equation \eqref{eq_BDEv}. We present its proof for the sake of completeness.

We first deduce an explicit expression for the reweighted process $Z_te^{\xi_t}$, for $t\ge 0$. In order to do so, we consider the function $f(x,y)= xe^{-y}$ and apply It\^o's formula (see e.g. Theorem II.5.1 in \cite{ikeda2014stochastic}). Then observing that $f_x'(x,y) = e^{-y}, f_y'(x,y)= -xe^{-y}, f_{xy}''(x,y)=f_{yx}''(x,y)=-e^{-y}, f_{xx}''(x,y)=0, f_{yy}''(x,y)= xe^{-y}$ and 
 applying directly Theorem II.5.1 in \cite{ikeda2014stochastic}, we see
\[
\begin{split}
	Z_te^{-\xi_t} =& Z_0 e^{-\xi_0} +\sigma\int_0^t f_x'(Z_s, \xi_s) Z_s\mathrm{d} B^{(e)}_s +\sigma\int_0^t f_y'(Z_s, \xi_s) \mathrm{d} B^{(e)}_s\\
	&+(\alpha+\delta)\int_0^t f_x'(Z_s, \xi_s) Z_s\mathrm{d} s+ \overline{\alpha}\int_0^t f_y'(Z_s, \xi_s)\mathrm{d} s+\frac{\sigma^2}{2}\int_0^t f_{xx}^{''}(Z_s, \xi_s) Z_s^2\mathrm{d} s \\
	&+\frac{\sigma^2}{2}\int_0^t f_{yy}^{''}(Z_s, \xi_s) \mathrm{d} s+\frac{\sigma^2}{2}\int_0^t f_{xy}^{''}(Z_s, \xi_s) Z_s\mathrm{d} s+
	\frac{\sigma^2}{2}\int_0^t f_{yx}^{''}(Z_s, \xi_s) Z_s\mathrm{d} s\\
	&+\int_0^t\int_{[0,\infty)}\int_0^{Z_{s-}}\left(f(Z_{s-}+z, \xi_{s-})-f(Z_{s-}, \xi_{s-})\right) N^{(b)}(\mathrm{d} s, \mathrm{d} z, \mathrm{d} u)\\
	&+\int_0^t\int_{(-1,1)^c}\left(f(Z_{s-}+Z_{s-}(e^x-1), \xi_{s-}+x)-f(Z_{s-}, \xi_{s-})\right) N^{(e)}(\mathrm{d} s, \mathrm{d} x)\\
	&+\int_0^t\int_{(-1,1)}\left(f(Z_{s-}+Z_{s-}(e^x-1), \xi_{s-}+x)-f(Z_{s-}, \xi_{s-})\right) \widetilde{N}^{(e)}(\mathrm{d} s, \mathrm{d} x)\\
	&+\int_0^t\int_{(-1,1)}\left(f(Z_{s}+Z_{s}(e^x-1), \xi_{s}+x)-f(Z_{s}, \xi_{s})\right.\\
	&\hspace{6cm}\left.-Z_s(e^x-1)f^{'}_x(Z_s, \xi_{s})-xf^{'}_y(Z_s, \xi_{s})\right)\pi(\mathrm{d} x)  \mathrm{d} s.
\end{split}
\]
By replacing all the factors in the above identity and recalling the definition of $\overline{\alpha}$ below identity \eqref{eq_envir2}, we obtain  
	\begin{equation}\label{eq:zexi}
Z_te^{-\xi_t} = Z_0e^{-\xi_0} + \int_{0}^{t} \int_{[0, \infty)} \int_{0}^{Z_{s-}} ze^{-\xi_{s-}}N^{(b)}(\mathrm{d}s, \mathrm{d}z, \mathrm{d} u).
	\end{equation}
	Next, we fix $\lambda\ge 0$ and $t\ge s\ge 0$.  We assume for now that $(v_t(s,\lambda , \xi), s\in [0,t])$ is  solution of the backward differential equation \eqref{eq_BDEv}. 
For simplicity on exposition, we denote by $H_t(s)= \exp\{-Z_se^{-\xi_s}v_t(s,\lambda, \xi)\}$ and we apply  again It\^o's formula  to the processes $(Z_se^{-\xi_s}, s\le t)$ and $(v_t(s,\lambda, \xi), s\le t)$ with the function $f(x,y)=e^{-xy}$,
that is
\begin{align*}
	H_t(t) & = H_t(s) + \int_{s}^{t}f_y'(Z_{r}e^{\xi_{r}}, v_t(r, \lambda,\xi)) \mathrm{d}(v_t(r, \lambda,\xi)) \\ & +  \int_{s}^{t} \int_{[0,\infty)} \int_{0}^{Z_{s-}} \left(f\Big((Z_{r-}+z)e^{-\xi_{r-}}, v_t(r,\lambda,\xi)\Big)
	\right.\\
	&\hspace{7cm}- f(Z_{r-}e^{\xi_{r-}}, v_t(r, \lambda,\xi)) \Big) N^{(b)}(\mathrm{d}r, \mathrm{d}z, \mathrm{d} u).
\end{align*}
Now  using  \eqref{eq_BDEv}, we get 
\begin{equation*}
	\begin{split}
		H_t(t) &= H_t(s) - \int_{s}^{t} H_t(r) Z_r e^{-\xi_r} e^{\xi_r} \psi_0\big(v_t(r,\lambda ,\xi)e^{-\xi_r}\big) \mathrm{d} r  \\ 
		&\, \hspace{3cm}+ \int_{s}^{t} \int_{[0,\infty)} \int_{0}^{Z_{r-}} H_t(r-) \Big(e^{-ze^{-\xi_{r-}}v_t(r,\lambda, \xi)}-1 	\Big) N^{(b)}(\mathrm{d}r, \mathrm{d}z, \mathrm{d} u) \\ 
		&=  H_t(s) - \int_{s}^{t} \int_{[0,\infty)}H_t(r) Z_r \Big(e^{-ze^{-\xi_{r}}v_t(r,\lambda, \xi)}-1 	\Big) \mu(\mathrm{d} z)\mathrm{d} r \\ 
		&\, \hspace{3cm}+ \int_{s}^{t} \int_{[0,\infty)} \int_{0}^{Z_{r-}} H_t(r-) \Big(e^{-ze^{-\xi_{r-}}v_t(r,\lambda, \xi)}-1 	\Big) N^{(b)}(\mathrm{d}r, \mathrm{d}z, \mathrm{d} u) \\
		 & = H_t(s) - \int_{s}^{t} \int_{[0,\infty)} \int_{0}^{Z_{s-}} H_t(r-)  \Big(1-e^{-ze^{-\xi_{r-}}v_t(r,\lambda, \xi)}	\Big)  \widetilde{N}^{(b)}(\mathrm{d}s, \mathrm{d} z, \mathrm{d} r),
	  \end{split}
\end{equation*} 
where $\widetilde{N}^{(b)}(\mathrm{d} s, \mathrm{d} z, \mathrm{d}r)$ denotes the compensated version of $N^{(b)}(\mathrm{d} s, \mathrm{d} z, \mathrm{d}r)$. By taking conditional  expectations in both sides we get  \[\mathbb{E}_{(z,x)}\left[H_t(t) \Big | \xi, \mathcal{F}^{(b)}_s\right] = H_t(s),\]
as expected.
 			 
 			 In order to show the existence of equation \eqref{eq_BDEv}, we will appeal to the extended version of  Carath\'eodory's existence Theorem \ref{teo_caratheo}. Fix $\omega \in \Omega^{(e)}$ and $t, \lambda\geq 0$.  Denote by $f: [0,t] \times \mathbb{R} \to [-\infty,0]$ the following function 
 			$$ f(s,\theta)= e^{\xi_s(\omega)}\psi_0\big(\theta e^{-\xi_s(\omega)}\big).
			$$ 
 			In the following we omit the notation $\omega$ for the sake of brevity. First, we observe that the mapping $s\mapsto f(s,\theta)$ is measurable for each fixed $ \theta \in \mathbb{R}$. Indeed, the latter follows from the fact that $\xi=(\xi_s,\ s \geq 0)$ possesses c\`adl\`ag paths and is $(\mathcal{F}^{(e)}_t)_{t\geq 0}$-adapted. More precisely, the application $(s,\omega)\mapsto \xi_s$ is $\mathcal{B}([0,t])\otimes \mathcal{F}^{(e)}_t$-measurable implying  that the mapping $s\mapsto f(s, \theta)$ is $\mathcal{B}([0,t])$-measurable for each fixed $\theta \in \mathbb{R}$ and $\omega\in \Omega^{(e)}$. Furthermore,  we have that the function $\theta \mapsto f(s,\theta)$ is continuous for each fixed $s\in [0,t]$ since $\psi_0$ is continuous. Therefore, according to Theorem \ref{teo_caratheo}, the proof is completed once we have shown that there exists an integrable function $m$ on $[0,t]$, such that, for any $(s,\theta) \in [0,t]\times \mathbb{R}$
 			\begin{equation*}
 				|f(s,\theta)|\leq m(s)\big(1+ |\theta|\big).
 			\end{equation*}
Note that, for any $(s,\theta) \in [0,t]\times \mathbb{R}$, we have
 			\begin{eqnarray*}
 				|f(s,\theta)| &=& \big|e^{\xi_s}\psi_0\big(\theta e^{-\xi_s}\big)\big| = e^{\xi_s}|\psi_0\big( \theta e^{-\xi_s}\big)|\mathbf{1}_{\{\xi_s\geq 0\}} + e^{\xi_s}|\psi_0\big(\theta e^{-\xi_s}\big)|\mathbf{1}_{\{\xi_s < 0\}}\\ &\leq & e^{\xi_s}|\psi_0(\theta)|\mathbf{1}_{\{\xi_s\geq 0\}} +e^{\xi_s}|\psi_0\big(\theta e^{-\xi_s}\big)|\mathbf{1}_{\{\xi_s < 0\}},
 			\end{eqnarray*}
 			where in the last inequality, we have used that $|\psi_0|$ is an increasing function. Now, since $|\psi_0|$ is a concave function, it is well-known that for any $\theta>0$ and $k>1$, we have $|\psi_0(\theta)|\leq k|\psi_0(\theta/k)|$ (see for instance the proof of  \cite[Proposition III. 1]{bertoin1996levy}). In particular, this inequality implies,
 			\[
 				|f(s,\theta)| \leq  e^{\xi_s}|\psi_0(\theta)|\mathbf{1}_{\{\xi_s\geq 0\}} +|\psi_0(\theta)|\mathbf{1}_{\{\xi_s < 0\}}  \leq \max\{e^{\xi_s}\mathbf{1}_{\{\xi_s\geq 0\}}, \mathbf{1}_{\{\xi_s < 0\}}\}|\psi_0(\theta)|. 
 			\]
 			On the other hand, since $|\psi_0|$ is a Bernstein function, it is well known that  there exists $\mathtt{c}, d>0$, such that $|\psi_0(\theta)|\leq  \mathtt{c}+d\theta$ for any $\theta \ge 0$ (see for instance, Corollary 3.8 in \cite{ScSVon}). It turns out that, 
 			\begin{eqnarray*}
 				|f(s,\theta)| &\leq & m(s) \big(1 + |\theta|\big) ,\quad \quad \text{for all} \quad  \quad (s,\theta) \in [0,t]\times \mathbb{R},
 			\end{eqnarray*}
 			where
 			$$m(s):=(\mathtt{c}\vee d)\max\{e^{\xi_s}\mathbf{1}_{\{\xi_s\geq 0\}}, \mathbf{1}_{\{\xi_s < 0\}}\},  \quad \quad  s \in [0,t].$$  
 			Note that $m$ is an integrable function on $[0,t]$ since the L\'evy process $\xi$ has c\`adl\`ag paths. 
 			Finally, thanks  to Theorem \ref{teo_caratheo}, there exists  an a.s. solution of \eqref{eq_BDEv}.  		\end{proof}

 		The functional $v_t(s,\lambda,\xi)$ has useful monotonicity properties as it is stated in the following lemma. In the forthcoming sections, we will make use of these properties.

 		\begin{lemma}\label{prop_monotonia}
 			For any $\lambda\geq 0$ and $t\geq 0$, the mapping $s \mapsto v_t(s,\lambda, \xi)$ is decreasing on $[0,t]$. For any $s\in [0,t]$, the mapping $\lambda \mapsto v_t(s,\lambda,\xi)$ is increasing on $[0,\infty)$.
 		\end{lemma}
 		
 		\begin{proof}
 			Recall that  $\psi_0(\theta)\leq 0$. Then, from the backward differential equation \eqref{eq_BDEv}, we see that the function  $s \mapsto v_t(s,\lambda,\xi)$ is decreasing on $[0,t]$. Further, from \eqref{eq_quenchedlaw} we observe that the mapping  $\lambda \mapsto v_t(s,\lambda,\xi)$ is increasing on $[0,\infty)$ as required.
 		\end{proof}
 		
 		We conclude this section with the proof  of Proposition \ref{lem_condsufnec}. Before we do so, let us make an important remark about the non-explosion and extinction  probabilities. It is easy to deduce, by letting $\lambda\downarrow 0$ in \eqref{eq_quenchedlaw} and  with the help of  the Monotone Convergence Theorem, that the non-explosion probability  is given by
 		\begin{equation}\label{eq_exploproba}
 			\mathbb{P}_{(z,x)}\big(Z_t<\infty \ | \ \xi\big) = \exp\left\{-z\lim\limits_{\lambda \downarrow 0} v_t(0,\lambda e^{-x},\xi-x)\right\}, \quad z, t >0, \quad  x\in \mathbb{R}.
 		\end{equation} 
 		With this  in hand, we may now  observe that the process $Z$ is \textit{conservative}  if and only if 
 		\begin{equation}\label{eq_lim_v}
 			\lim\limits_{\lambda \downarrow 0} v_t(0,\lambda e^{-x},\xi-x)=0, \quad \quad  \text{for all} \quad \quad  t>0.
 		\end{equation}
 	On the other hand we also observe that, by letting $\lambda\uparrow \infty$ in \eqref{eq_quenchedlaw} and  using again the Monotone Convergence Theorem,  the  probability of extinction is such that
		\[
 			\mathbb{P}_{(z,x)}\big(Z_t=0 \ | \ \xi\big) = \exp\left\{-z\lim\limits_{\lambda \uparrow \infty} v_t(0,\lambda e^{-x},\xi-x)\right\}, \quad z, t >0, \quad  x\in \mathbb{R}.
 		\]
		Moreover, from Lemma \ref{prop_monotonia} and since $v_t(t,\lambda e^{-x},\xi-x)=\lambda e^{-x}$, we have that 
		\[
		v_t(0,\lambda e^{-x},\xi-x)\ge \lambda e^{-x}, \qquad \textrm{for }\quad \lambda, t> 0,
		\] which combined with the previous identity clearly implies that 
		\begin{equation}\label{survprob}
		\mathbb{P}_{(z,x)}\big(Z_t>0 \ | \ \xi\big) =1,\quad \quad  \text{for all} \quad \quad  z, t>0.
		\end{equation}

 		\begin{proof}[Proof of Proposition \ref{lem_condsufnec}]
 			Fix $t>0$ and recall from \eqref{eq:runsupinf} that  $(\underline{\xi}_t, t\geq 0)$ and $(\overline{\xi}_t, t\geq 0)$ denote the running infimum  and supremum of the process $\xi$, respectively. First, we assume that the pure branching mechanism $\psi_0$ satisfies \eqref{eq_Grey}, that is to say, 
 			$$ \int_{0+}\frac{1}{|\psi_0(z)|}\mathrm{d} z =  \int_{0+}\frac{1}{-\psi_0(z)}\mathrm{d} z = \infty.$$
 			From Theorem \ref{teo_existencia}, we see that the backward differential equation \eqref{eq_BDEv} can be rewritten as follows
 			\begin{equation*}
 				-t=\int_{0}^{t} \frac{\mathrm{d} v_t(s,\lambda ,\xi)}{-e^{\xi_s}\psi_0\big(e^{-\xi_s}v_t(s,\lambda,\xi)\big)}  =  \int_{0}^{t} \frac{\mathrm{d} v_t(s,\lambda,\xi)}{e^{\xi_s}|\psi_0\big(e^{-\xi_s}v_t(s,\lambda,\xi)\big)|}.
 			\end{equation*}
 			Now, we recall that $|\psi_0|$ is an increasing and non-negative function. Then appealing to the definition of the running infimum  and supremum of $\xi$,  we observe that the following inequality holds 
 			\[\begin{split}
 				-t = \int_{0}^{t} \frac{\mathrm{d}v_t(s,\lambda,\xi)}{e^{\xi_s}|\psi_0\big(e^{-\xi_s}v_t(s,\lambda,\xi)\big)|}&\leq \int_{0}^{t} \frac{\mathrm{d}v_t(s,\lambda,\xi)}{e^{\underline{\xi}_t} |\psi_0\big(e^{-\overline{\xi}_t}v_t(s,\lambda,\xi)\big)|}\\ &= \frac{e^{\overline{\xi}_t}}{e^{\underline{\xi}_t}} \int_{e^{-\overline{\xi}_t} v_t(0,\lambda, \xi)}^{e^{-\overline{\xi}_t} \lambda} \frac{1}{|\psi_0(z)|}\mathrm{d} z,
 			\end{split}	\]
 			where in the last equality we have used the change of variables  $z= e^{-\overline{\xi}_t}v_t(s,\lambda,\xi)$. 
 			Next, letting $\lambda\downarrow 0$ in the previous inequality, we get
 			\begin{equation}\label{eq_2}
 				\frac{e^{\overline{\xi}_t}}{e^{\underline{\xi}_t}} \int_{0}^{e^{-\overline{\xi}_t} \lim\limits_{\lambda \downarrow 0}v_t(0,\lambda,\xi)} \frac{1}{|\psi_0(z)|}\mathrm{d} z \leq t.
 			\end{equation}
 			Thus, taking into account our assumption, we are forced to conclude that   
 			\begin{equation}\label{eq_limitevxi}
 				\lim_{\lambda \downarrow 0} v_t(0,\lambda ,\xi)=0.
 			\end{equation}
 			In other words, the process is conservative. 
 			
 			On the other hand, we assume that the process is conservative or equivalently that \eqref{eq_limitevxi} holds. We will proceed by contradiction, thus we suppose that 
 			$$ \int_{0+}\frac{1}{|\psi_0(z)|}\mathrm{d} z  < \infty.$$
 			Similar to  the  above arguments, we deduce that 
 			\[\begin{split}
 				-t = \int_{0}^{t} \frac{\mathrm{d} v_t(s,\lambda,\xi)}{-e^{\xi_s}\psi_0\big(e^{-\xi_s}v_t(s,\lambda,\xi)\big)}&\geq \int_{0}^{t} \frac{\mathrm{d} v_t(s,\lambda,\xi)}{e^{\overline{\xi}_t} |\psi_0\big(e^{-\underline{\xi}_t}v_t(s,\lambda,\xi)\big)|}\\&= \frac{e^{\underline{\xi}_t}}{e^{\overline{\xi}_t}} \int_{e^{-\underline{\xi}_t} v_t(0,\lambda, \xi)}^{e^{-\underline{\xi}_t}\lambda} \frac{1}{|\psi_0(z)|}\mathrm{d} z. 
 			\end{split}\]
 			Taking $\epsilon>0$ sufficiently small, we see 
 			\begin{equation*}\label{eq_1}
 				t \leq  \frac{e^{\underline{\xi}_t}}{e^{\overline{\xi}_t}}  \int_{e^{-\underline{\xi}_t} \lambda}^{\epsilon} \frac{1}{|\psi_0(z)|}\mathrm{d} z - \frac{e^{\underline{\xi}_t}}{e^{\overline{\xi}_t}} \int_{e^{-\underline{\xi}_t} v_t(0,\lambda, \xi)}^{\epsilon} \frac{1}{|\psi_0(z)|}\mathrm{d} z.
 			\end{equation*}
 			Hence, by taking $\lambda\downarrow 0$ in the above inequality and using \eqref{eq_limitevxi}, we have $t\leq 0$,  which is a contradiction. Therefore, we deduce that the pure branching mechanism $\psi_0$ satisfies \eqref{eq_Grey}.
 		\end{proof}

 		\section{Subcritical explosive regime: proof of Theorem \ref{teo_explosubcritica}}\label{sub_explosionsub}
 			 The proof of Theorem \ref{teo_explosubcritica} follows similar ideas as those used in the proof of Proposition 3 in Palau and Pardo \cite{palau2018branching}.
 			\begin{proof}[Proof of Theorem \ref{teo_explosubcritica}]
 				Let $z>0$ and $x\in \mathbb{R}$. We begin by observing, from   \eqref{CBILRE} or \eqref{eq_exploproba},  that $\mathbb{P}_{(z,x)}(Z_t<\infty)$ does not depend of the initial value $x$ of the L\'evy process $\xi$. More precisely,  from  \eqref{CBILRE} we can observe that $Z_t$ depends only on the initial condition $Z_0=z$. Thus, again from \eqref{eq_exploproba}, we see
 				\begin{equation*}
 					\lim\limits_{t\to \infty} \mathbb{P}_{z}(Z_t<\infty) =  \lim\limits_{t\to \infty}  \mathbb{E}^{(e)}_{x}\Big[\exp\left\{-zv_t(0,0,\xi-x)\right\}\Big].
 				\end{equation*}
 				 Hence, as soon as we can establish that
 				\begin{equation*}
 					\lim_{t\to \infty} v_t(0,0,\xi-x) < \infty, \quad  \mathbb{P}_{x}^{(e)}-\text{a.s.},
 				\end{equation*}
 				our proof is completed.
				  
				 				From Lemma \ref{prop_monotonia}, we see that the mapping $s\mapsto v_t(s,\lambda e^{-x}, \xi-x)$ is decreasing  and since  $v_t(t,\lambda e^{-x},\xi- x)=\lambda e^{-x}$, we have $v_t(s,\lambda e^{-x},\xi-x)\geq \lambda e^{-x}$ for all $s\in[0,t]$. It follows that, for $s\in[0,t]$ and $\lambda> 0$,
 				\[\begin{split}
 					\frac{\partial }{\partial s} v_t(s,\lambda e^{-x},& \xi-x) =e^{\xi_s-x} \psi_0\big(v_t(s,\lambda e^{-x},\xi -x)e^{-\xi_s+x}\big)\\ &= -v_t(s,\lambda e^{-x},\xi-x)\int_{0}^\infty \exp\{-v_t(s,\lambda e^{-x}, \xi- x)e^{-\xi_s+x} z\}  \bar{\mu}(z)\mathrm{d}z \\ &\geq -v_t(s,\lambda e^{-x},\xi-x)\int_{0}^\infty \exp\{- \lambda e^{-x} e^{-\xi_s+x} z\}  \bar{\mu}(z)\mathrm{d} z.
 				\end{split}\]
 				Therefore by integrating,
 				\begin{equation*}
 					\log v_t(0,\lambda e^{-x},\xi-x)\leq \log (\lambda e^{-x}) +\int_{0}^{t} \int_{0}^\infty \exp\{- \lambda e^{-\xi_s} z\}  \bar{\mu}(z)\mathrm{d} z  \mathrm{d} s.
 				\end{equation*}
 				Moreover, from Lemma \ref{prop_monotonia}, we have that the mapping $\lambda e^{-x} \mapsto v_t(s,\lambda e^{-x},\xi-x)$ is increasing. It tuns out that,
 				\begin{equation}\label{eq:boundv}
 					v_t(0,0,\xi-x)\leq v_t(0,\lambda e^{-x},\xi -x)\leq \lambda e^{-x} \exp\left(\int_{0}^{t} \int_{0}^\infty \exp\{- \lambda e^{-\xi_s} z\}  \bar{\mu}(z)\mathrm{d} z  \mathrm{d} s\right).
 				\end{equation}
 				Using the definition of $\Phi_\lambda$ in \eqref{eq_phihat}, we get
 		\begin{equation}\label{eq_vphihatint}
		\lim\limits_{t\to \infty}v_t(0,0,\xi-x) \leq \lambda e^{-x}\exp\left(\int_{0}^{\infty}\Phi_\lambda(-\xi_s) \mathrm{d} s\right).
\end{equation}
 We recall that  in this regime the process $-\xi$ drifts to $\infty$, $\mathbb{P}^{(e)}_x$ -a.s. Thus, in order to prove that the integral in \eqref{eq_vphihatint} is finite,  let us introduce $\varsigma =\sup\{t\geq 0: -\xi_t\leq 0\}$ and observe that
 				\begin{equation*}\label{eq_2integrals}
 					\int_{0}^{\infty}\Phi_\lambda(-\xi_s) \mathrm{d} s = \int_{0}^{\varsigma}\Phi_\lambda(-\xi_s) \mathrm{d} s\  + \  \int_{\varsigma}^{\infty}\Phi_\lambda(-\xi_s) \mathrm{d} s.
 				\end{equation*}
 				Since $\varsigma <\infty$, $\mathbb{P}^{(e)}_x$ -a.s., it follows that the first integral in the right-hand side above is finite $\mathbb{P}^{(e)}_x$-a.s. For the second integral, we may appeal to Theorem 1 in Erickson and Maller \cite{maller} or the main result in Kolb and Savov \cite{KolSav}. It is important to note that the main result  in \cite{KolSav} extends the result in \cite{maller} but both results coincide in this particular case.  Despite that the aforementioned Theorem is written in terms of L\'evy processes drifting to $+\infty$, it is not so difficult to see that it can be rewritten, using duality, in terms of L\'evy processes drifting to $-\infty$ which is how we are using it  here.  More precisely, since $\Phi_\lambda$ is a finite positive, non-constant and non-increasing function on $[0,\infty)$ and  condition \eqref{eq_Axi} holds, Theorem 1 in \cite{maller} guarantees  that 
 				\[\int_{\varsigma}^{\infty}\Phi_\lambda(-\xi_s) \mathrm{d} s<\infty, \qquad \mathbb{P}^{(e)}_x-\text{a.s.}\]
Furthermore, if  $\mathbb{E}^{(e)}\big[\xi_1\big]\in (-\infty,0)$, then $\lim\limits_{x\to \infty}A_{\xi}(x)$ is finite. In particular, it follows by integration by parts, that the integral condition \eqref{eq_Axi} is equivalent to 
 				\[\int_{0}^{\infty} \Phi_\lambda(u)\mathrm{d} u <\infty.\]
 				Moreover, we have
 				\[\begin{split}
 					\int_{0}^{\infty} \Phi_\lambda(u)\mathrm{d} u & = \int_{0}^{\infty} \int_{0}^\infty \exp\{- \lambda e^{u} y\}  \bar{\mu}(y)\mathrm{d} y\mathrm{d} u = \int_{0}^{\infty} \int_{1}^{\infty} \frac{\exp\{- \lambda y w \}}{w}  \mathrm{d} w\bar{\mu}(y)\mathrm{d} y.
 				\end{split}\]
 				Now by the definition of the exponential integral given in \eqref{eq_E1def}, we deduce that condition \eqref{eq_Axi} is equivalent to 
 				\[\int_{0}^{\infty}  \normalfont{\texttt{E}_1}(\lambda y) \bar{\mu}(y)\mathrm{d} y<\infty,\]
 				which concludes the proof. 
 			\end{proof}
 			\textcolor{red}{}

 		\section{CSBPs in a conditioned L\'evy environment}\label{sec_explosionconditioned}
			Throughout this section, we shall suppose that the L\'evy process $\xi$ satisfies Spitzer's condition \eqref{eq_spitzer} which in particular implies that the process oscillates. 
			
			As mentioned earlier, the asymptotic behaviour of  the non-explosion probability is related to fluctuations of the L\'evy process $\xi$, specially to its running supremum.  For our purpose, we need to introduce the definition of a CSBP conditioned to stay negative, which roughly speaking means that the running supremum of the auxiliary L\'evy process $\xi$ is negative. Let us therefore spend some time in this section gathering together some of the facts of this conditioned version.

 	Similarly to the definition of L\'evy processes conditioned to stay positive  and following a similar strategy as in the discrete framework in Afanasyev et al. \cite{afanasyev2005criticality}, we would like to introduce a continuous-state branching process in a L\'evy environment conditioned to stay negative as a Doob-$h$ transform. The aforementioned process was first investigated by Bansaye et al.  \cite{bansaye2021extinction} with the aim to study the survival event in a critical L\'evy environment. 
			
	 		\begin{lemma}[Bansaye et. al. \cite{bansaye2021extinction}]\label{teo_bansayemtg}
 			Let $z,x >0$. The process $\{\widehat{U}(\xi_t)\mathbf{1}_{\{\underline{\xi}_t> 0\}}, t\geq 0\}$ is a martingale with respect to $(\mathcal{F}_t)_{t\geq 0}$ and under $\mathbb{P}_{(z,x)}$.
 		\end{lemma}
 		With this in hand, they introduce the law of a \textit{continuous-state branching process  in a L\'evy environment $\xi$ conditioned to stay positive} as follows, for $\Lambda \in \mathcal{F}_t$, $z,x>0$, 
 		\begin{equation*}
 			\mathbb{P}^{\uparrow}_{(z,x)}(\Lambda):=\frac{1}{\widehat{U}(x)}\mathbb{E}_{(z,x)}\big[\widehat{U}(\xi_t)\mathbf{1}_{\{\underline{\xi}_t> 0\}}\mathbf{1}_\Lambda\big],
 		\end{equation*}
 		where $\widehat{U}$ is the renewal function defined in \eqref{eq_renewalfns}. It is natural therefore to cast an eye on similar issues for the study of non-explosion events in a  L\'evy environment. In contrast, we introduce here the process $Z$ in a L\'evy  environment $\xi$ conditioned to stay negative. Recall that $\widehat{\xi}$ is the dual process of $\xi$. 
 		Appealing to duality and Lemma \ref{teo_bansayemtg}, we can see that the process $\{U(-\xi_t)\mathbf{1}_{\{\overline{\xi}_t< 0\}}, t\geq 0\}$ is a martingale with respect to $(\mathcal{F}_t)_{t\geq 0}$ and under $\mathbb{P}_{(z,x)}$ with $z>0$ and $x<0$. Then we introduce the law of the \textit{continuous-state branching process in a L\'evy environment $\xi$ conditioned to stay negative}, as follows: for $\Lambda \in \mathcal{F}_t$ for $z>0$ and $x<0$, 
 		
 		\begin{equation}\label{eq_CDBPnegativo}
 			\mathbb{P}^{\downarrow}_{(z,x)}(\Lambda): =\frac{1}{U(-x)}\mathbb{E}_{(z,x)}\big[U(-\xi_t)\mathbf{1}_{\{\overline{\xi}_t< 0\}}\mathbf{1}_\Lambda\big].
 		\end{equation}
 		Intuitively speaking, $\mathbb{P}^{\uparrow}_{(z,x)}$ and $\mathbb{P}^{\downarrow}_{(z,x)}$ correspond to the law of $(Z,\xi)$ conditioning the random environment $\xi$  to not enter $(-\infty,0)$ and $(0,\infty)$, respectively.

 		
 		The following convergence result is crucial for Theorem \ref{teo_explocritica}. 
 		The proof can be obtained directly using duality and Lemma 3.2 in \cite{bansaye2021extinction}.
 		\begin{lemma}\label{prop_maxR}
 			Fix  $z>0,\  x<0$	and assume that Spitzer's condition \eqref{eq_spitzer} holds. Let $R_s$ be a bounded real-valued $\mathcal{F}_s$-measurable random  variable. Then
 			\begin{equation*}
 				\lim\limits_{t\to \infty} \mathbb{E}_{(z,x)} \big[R_s \ | \ \overline{\xi}_t < 0 \big]= \mathbb{E}^{ \downarrow}_{(z,x)}\big[R_s\big].
 			\end{equation*}
 			More generally, let $(R_t, t\geq 0)$ be a uniformly bounded  real-valued  process adapted to the filtration $(\mathcal{F}_t, t\geq 0)$, which converges $\mathbb{P}^{\downarrow}_{(z,x)}$-a.s., to some random variable $R_\infty$. Then
 			\begin{equation*}
 				\lim\limits_{t\to \infty}\mathbb{E}_{(z,x)}\big[R_t \ | \ \overline{\xi}_t < 0\big] = \mathbb{E}^{ \downarrow}_{(z,x)}\big[R_\infty\big].
 			\end{equation*}
 		\end{lemma}

 			Recall from Theorem \ref{teo_existencia} that the quenched law of the process $(Z_t e^{-\xi_t}, t\geq 0)$ is completely characterised by the functional $v_t(s,\lambda,\xi)$. In the case of conditioned environment we have a similar result. We formalise this in the following lemma whose proof essentially mimics the steps of Proposition 3.3 in  \cite{bansaye2021extinction} and  identity  \eqref{survprob}. 
 			
 			\begin{lemma}\label{lem_leynegativa}
 				For each $z>0$, $x<0$ and $\lambda \geq 0$, we have 
 				\begin{equation}
 					\mathbb{E}_{(z,x)}^{\downarrow}\Big[\exp\big\{-\lambda Z_t e^{-\xi_t}\big\}\Big] = \mathbb{E}_x^{(e),\downarrow}\Big[\exp\{-z v_t(0,\lambda e^{-x} ,\xi-x)\}\Big].
 				\end{equation}
 				In particular, 
 				\begin{equation*}
 					\mathbb{P}^{\downarrow}_{(z,x)}(Z_t<\infty) =   \mathbb{E}^{(e),\downarrow}_{x}\Big[\exp\left\{-zv_t(0,0,\xi-x)\right\}\Big],
 				\end{equation*}
				and 
			\begin{equation*}
 					\mathbb{P}^{\downarrow}_{(z,x)}(Z_t>0) =  1.
 				\end{equation*}
 			\end{lemma}
 			
 			
 			
 			The following lemma states that, with respect to $ \mathbb{P}^{\downarrow}_{(z,x)}$, the population has positive probability to be finite forever. In other words, $Z$ has a positive probability to be finite when the running supremum of the L\'evy environment is negative.  
			 			The statement holds under the  moment condition \eqref{eq_Hyp_SubH} of the L\'evy measure $\mu$.

 			\begin{lemma}\label{lem_proba_ext}
 				Assume that the L\'evy process $\xi$ satisfies Spitzer's condition \eqref{eq_spitzer} and condition \eqref{eq_Hyp_SubH}. Then,  for $z >0$ and $x<0$, we have
 				\begin{equation*}
 					\lim\limits_{t\to \infty} \mathbb{P}^{\downarrow}_{(z,x)}\left(Z_t<\infty\right) > 0.
 				\end{equation*}
 			\end{lemma}
			Note that such behaviour is similar to the behaviour in the subcritical explosive regime (i.e. when the environment drifts to $-\infty$) given in Theorem \ref{teo_explosubcritica}.
 			
 			\begin{proof}
 				Let $z>0$ and $x<0$. From Lemma \ref{lem_leynegativa}, we already know the formula,
 				\begin{equation*}
 					\lim\limits_{t\to \infty} \mathbb{P}^{\downarrow}_{(z,x)}(Z_t<\infty) =  \lim\limits_{t\to \infty}  \mathbb{E}^{(e),\downarrow}_{x}\Big[\exp\left\{-zv_t(0,0,\xi-x)\right\}\Big].
 				\end{equation*}
 				Then, similarly as in Theorem \ref{teo_explosubcritica}, in order to deduce our result  we will  show that 
 				\begin{equation}\label{eq_vfinite}
 					\lim_{t\to \infty} v_t(0,0,\xi-x) < \infty, \qquad  \mathbb{P}_{x}^{(e),\downarrow}-\text{a.s.}
 				\end{equation}
 				We recall from the proof of Theorem \ref{teo_explosubcritica} that for all $\lambda>0$
 				\[
				\lim\limits_{t\to \infty}v_t(0,0,\xi-x) \leq \lambda e^{-x}\exp\left(\int_{0}^{\infty} \Phi_\lambda(-\xi_s) \mathrm{d} s\right).
				\]
 				Now, if the right-hand side of the above inequality is finite $\mathbb{P}_{x}^{(e),\downarrow}-\text{a.s.}$ then \eqref{eq_vfinite} holds. 
			The result is thus proved once we show that
				\begin{equation*}
 					\mathbb{E}_{x}^{(e),\downarrow}\left[\int_{0}^{\infty}\Phi_\lambda(-\xi_s) \mathrm{d} s\right] < \infty.
				\end{equation*}
 				First, with the help of Fubini's Theorem and the definition of the measure $\mathbb{P}_{x}^{(e),\downarrow}$ in terms of $\widehat{\mathbb{P}}_{x}^{(e)}$  (the law of the dual of $\xi$) we obtain 
				\[
				\begin{split}
					\mathbb{E}_{x}^{(e),\downarrow}\left[\int_{0}^{\infty}\Phi_\lambda(-\xi_s) \mathrm{d} s\right] &= \frac{1}{U(-x)}\int_{0}^{\infty} \widehat{\mathbb{E}}^{(e)}_{-x}\big[U(\xi_s)\Phi_\lambda(\xi_s)\mathbf{1}_{\{\underline{\xi}_s> 0\}}\big] \mathrm{d} s \\ &= \frac{1}{U(-x)} \widehat{\mathbb{E}}^{(e)}_{-x}\left[\int_{0}^{\tau_0^-}U(\xi_s)\Phi_\lambda(\xi_s)\mathrm{d} s\right].
			\end{split}
			\]
				Now, applying Theorem VI.20 in Bertoin \cite{bertoin1996levy} to the dual process $\widehat{\xi}=-\xi$ and the function $f(y)=U(y)\Phi_\lambda(y), \ y\geq 0$, we deduce that,  there exists a constant $k>0$ such that
				\begin{equation*}
					\widehat{\mathbb{E}}^{(e)}_{-x}\left[\int_{0}^{\tau_0^-}U(\xi_s)\Phi_\lambda(\xi_s)\mathrm{d} s\right]= k \int_{[0,\infty)} \mathrm{d} \widehat{U}(y)\int_{[0,-x]} \mathrm{d}U(z)U(y-x-z)\Phi_\lambda(y-x-z).
 				\end{equation*}
 				For the sake of simplicity we take $k=1$ (we may choose a  normalisation of the local time in order to have $k=1$). Observe that, for any $z\in[0,-x]$ and $y\geq 0$, we have $y-x-z\leq y-x$. Further, since $U(\cdot)$  and $\Phi_\lambda(\cdot)$ are increasing functions, we deduce that  $U(y-x-z)\leq U(y-x)$ and $\Phi_\lambda(y-x-z)\leq \Phi_\lambda(y)$, which implies
				\[
 				\begin{split}
 					\widehat{\mathbb{E}}^{(e)}_{-x}\left[\int_{0}^{\tau_0^-}U(\xi_s)\Phi_\lambda(\xi_s)\mathrm{d} s\right]&\leq  \int_{[0,\infty)} \mathrm{d} \widehat{U}(y)\int_{[0,-x]} \mathrm{d}U(z)U(y-x)\Phi_\lambda(y) \\ &= (U(-x)-U(0))\int_{[0,\infty)} \mathrm{d} \widehat{U}(y)U(y-x)\Phi_\lambda(y).
 				\end{split}
				\]
Next recall that we may rewrite  the function $\Phi_\lambda$  as follows,
 \[
 				\begin{split}
 					\Phi_\lambda(u)&= \int_{0}^\infty \exp\{- \lambda e^{u} z\}  \bar{\mu}(z)\mathrm{d} z   \\
					&= \frac{e^{-u}}{\lambda }\int_{(0,\infty)} \Big(1-\exp\{- \lambda e^{u} z\}\Big)  \mu(\mathrm{d}z),
 				\end{split}
				\]
				implying that 
				\[
 				\begin{split}
 					\Phi_\lambda(u)&\le \int_{(0,\frac{e^{-u}}{\lambda }]}z \mu(\mathrm{d}z)+  \frac{e^{-u}}{\lambda }\int_{(\frac{e^{-u}}{\lambda }, \infty)}\mu(\mathrm{d}z).
					 				\end{split}
				\]
Hence
		 \[
 				\begin{split}		
				\int_{[0,\infty)}\mathrm{d} \widehat{U}(y)U(y-x)\frac{e^{-y}}{\lambda }\int_{(\frac{e^{-y}}{\lambda }, \infty)}\mu(\mathrm{d}z)&\le \frac{ \bar{\mu}(\lambda^{-1})}{\lambda }\int_{[0,\infty)}\mathrm{d} \widehat{U}(y)U(y-x)e^{-y} \\
				& \leq C_1\frac{ \bar{\mu}(\lambda^{-1})}{\lambda }
					\int_{[0,\infty)}\mathrm{d} \widehat{U}(y)ye^{-y} \\
					&=C_1	\frac{\bar{\mu}(\lambda^{-1})}{\lambda }	\frac{\widehat{\kappa}^\prime(0,1)}{\widehat{\kappa}(0,1)}<\infty,
					\end{split}
				\]
				where in the second inequality and in the last identity we have used, respectively, inequality \eqref{grandO} and  identity \eqref{bivLap}, but in terms of $\widehat{U}$ and after taking the first derivative. Now, we consider
				 \[
 				\begin{split}		
				\int_{[0,\infty)}\mathrm{d} \widehat{U}(y)U(y-x)\int_{(0,\frac{e^{-y}}{\lambda }]}z \mu(\mathrm{d}z)&\le \int_{(0,\infty)}\mu(\mathrm{d}z)z \int_{[0,\infty)}\mathrm{d} \widehat{U}(y)U(y-x)\mathbf{1}_{\{z<e^{-y}\lambda^{-1}\}} \\
				& \leq  \int_{(0,\lambda^{-1})}\mu(\mathrm{d}z)z \int_{[0,-\ln(z\lambda))}\mathrm{d} \widehat{U}(y)U(y-x) \\
					&\le \int_{(0,\lambda^{-1})}\mu(\mathrm{d}z)z \widehat{U}(-\ln(z\lambda))U(-\ln(z\lambda)-x),
					\end{split}
				\]
which is clearly finite from condition \eqref{eq_Hyp_SubH}.
 				Hence putting all pieces together, we obtain
 				\[
 					\mathbb{E}_{x}^{(e),\downarrow}\left[\int_{0}^{\infty}\Phi_\lambda(\xi_s) \mathrm{d} s\right] <\infty.
 				\]
			This concludes the proof.
 			\end{proof}
 		It is important to note that in Baguley et al. \cite{BaDoKy} there is a necessary and sufficient condition (integral test) for the finiteness of path integrals for standard Markov processes, see Theorem 2.3 and comments below, in terms of their potential measures. The latter  is in line with our  necessary condition for the finiteness of 
			\[
			\int_{0}^{\infty}\Phi_\lambda(-\xi_s) \mathrm{d} s,
			\]
			under $\mathbb{P}_{x}^{(e),\downarrow}$.

 			\section{Critical explosive regime: proof of Theorem \ref{teo_explocritica}}\label{sec_explosiontheoremrates}
			Throughout this section, we shall suppose that the L\'evy process $\xi$ satisfies Spitzer's condition \eqref{eq_spitzer}.

			The strategy  of our proof follows similar arguments as in \cite{bansaye2021extinction}, where the extinction event has been considered,   that is we split the event $\{Z_t<\infty\}$ into two events by considering the behaviour of the running supremum of the environment. 
			More precisely,  we split the non-explosion event  into either unfavourable environments, i.e. when the running supremum is  negative, or favourable environments, i.e. when the running supremum is positive. 
 			
 			Before we prove Theorem  \ref{teo_explocritica}, we introduce several useful results. Lemmas  \ref{prop_maxR}  and \ref{lem_proba_ext} allow us to establish the following result which describes the limit of the non-explosion probability  when the associated environment is conditioned to be negative.  
 			
 			\begin{proposition}\label{prop_cotanocero}
 				Suppose that conditions  \eqref{eq_spitzer} and \eqref{eq_Hyp_SubH}  are satisfied. Then for every $z >0$ and $x<0$, there exists $0<c(z,x)<\infty$ such that
 				
 				\begin{equation}\label{eq_cota1}
 					\lim\limits_{t \to \infty}\frac{1}{\kappa(1/t,0)}\mathbb{P}_{(z,x)}\Big(Z_t<\infty, \ \overline{\xi}_t <  0\Big) = c(z,x)U(-x) .
 				\end{equation}
 			\end{proposition}
 			
 			\begin{proof}
 				We begin by defining  the decreasing sequence of events $A_t=\{Z_t<\infty\}$ for $t\geq 0$,  and also the event $A_\infty= \{\forall t \geq 0,\ Z_t<\infty\}$. Now, observe that
 				$$\lim\limits_{t\to \infty} A_t = A_\infty.$$ 
 				Let $(R_t:=\mathbf{1}_{A_t},\ t\geq 0)$ be a uniformly bounded process adapted to the filtration $(\mathcal{F}_t, t\geq 0)$.  Note that, the process $(R_t, t\geq 0)$ converges $\mathbb{P}^\downarrow_{(z,x)}$-a.s. to a random variable $R_\infty=\mathbf{1}_{A_\infty}$. Then, by appealing to Lemma \ref{prop_maxR}, we have 
 				\begin{equation}\label{eq_limt_R}
 					\lim\limits_{t\to \infty}\mathbb{E}_{(z,x)}\big[R_t\ | \ \overline{\xi}_t < 0\big] = \mathbb{E}^{ \downarrow}_{(z,x)}\big[R_\infty\big].
 				\end{equation}
 				Therefore, by using the asymptotic behaviour of the probability that the L\'evy process $\xi$ remains negative
 				given in \eqref{eq_lim_M}, we get
 				\begin{eqnarray*}
 					\mathbb{P}_{(z,x)}\Big(Z_t<\infty,\ \overline{\xi}_t <  0\Big) &=&  \mathbb{E}_{(z,x)}\big[R_t \ |\  \overline{\xi}_t < 0\big] \mathbb{P}_{x}^{(e)}\big(\overline{\xi}_t <  0\big)\\  &\sim& c(z,x)U(-x)\kappa(1/t,0), \quad \text{as}\quad t \to \infty,
 				\end{eqnarray*}
 				where $c(z,x):= \mathbb{E}_{(z,x)}^{\downarrow}\big[R_\infty\big]/\sqrt{\pi}$. Furthermore,  from Lemma \ref{lem_proba_ext}, we have
 				\begin{equation*}
 					\mathbb{E}_{(z,x)}^{\downarrow}\big[R_\infty\big] = \mathbb{P}_{(z,x)}^\downarrow\big(\forall t \geq 0,\ Z_t<\infty\big) =  \lim\limits_{t\to \infty} \mathbb{P}^{\downarrow}_{(z,x)}(Z_t<\infty) >0,
 				\end{equation*}
 				which completes the first claim.
 			\end{proof}

			Denote by $\tau^{-}_x$, the first time that $\xi$ is below $(-\infty, x)$, i.e. $\tau^{-}_x=\inf\{t\ge 0: \xi_t\le x\}$,  for $x<0$. The following result follows directly from the inequality  (4.7) in \cite{bansaye2021extinction} (see the proof of Lemma 4.2), duality, the identity $\widehat{\mathbb{P}}^{(e)}(  \tau^{-}_{w} >t-\epsilon   )=\widehat{\mathbb{P}}^{(e)}_{-w}( \underline{\xi}_{t-\epsilon}>0 )$ and the estimate in \eqref{eq_lim_M}. We present its proof for the sake of completeness. 
 			\begin{lemma}\label{lem_cota_P} 
 				Let $x<0$ and assume that condition \eqref{eq_spitzer} holds. Thus for any $s \leq t$, as $t$ and $s$ goes to $\infty$, we have
 				\begin{equation*}
				\begin{split}
 					\widehat{\mathbb{P}}^{(e)}\big(s<\tau_{x}^-\leq t\big) &\leq \left( C_2\left( \frac{t}{s}\right)^{\eta+ \rho}-1\right) \widehat{\mathbb{P}}^{(e)}\big(\tau_{x}^-> t\big)\\
					&\le C_3\left( C_2\left( \frac{t}{s}\right)^{\eta+ \rho}-1\right) \kappa(1/t,0) \frac{U(-x)}{\sqrt{\pi}},
					\end{split}
 				\end{equation*}
 				where $C_2>0$, $C_3>1$ and $\eta>0$.
 			\end{lemma}
 		\begin{proof}
 						Let $x<0$ and $s\leq t$. We begin by noting
 						\begin{eqnarray}\label{eq_cota_tau}
 							\widehat{\mathbb{P}}^{(e)}\big(s< \tau^-_{x}\leq t\big) &=& \widehat{\mathbb{P}}^{(e)}\big( \tau^-_{x}>s\big)-\widehat{\mathbb{P}}^{(e)}\big( \tau^-_{x}>t\big) \nonumber \\ 
 							&=&   \left(\frac{\widehat{\mathbb{P}}^{(e)}\big( \tau^-_{x}>s\big)}{\widehat{\mathbb{P}}^{(e)}\big( \tau^-_{x}>t\big) }-1 \right)\widehat{\mathbb{P}}^{(e)}\big( \tau^-_{x}>t\big).
 						\end{eqnarray}
 						Now, recall that under Spitzer's condition  \eqref{eq_spitzer}, the function $\kappa(\cdot, 0)$ is regularly varying at $0$ or more precisely,  from \eqref{eq_lim_M}, we have
 						\begin{equation*}
 							\widehat{\mathbb{P}}^{(e)}\big(\tau^-_{x}> t\big) \sim \frac{U(-x)}{\sqrt{\pi}} t^{-\rho}\ell(t), \quad \text{as} \quad t \to \infty,
 						\end{equation*}
 						where $\ell$ is  the slowly varying function at $\infty$ defined in \eqref{eq_k} as $\ell(t)=\widetilde{\ell}(1/t)$. Hence,  for $t$ and $s$ large enough, we have
 						\begin{equation*}
 							\frac{\widehat{\mathbb{P}}^{(e)}\big( \tau^-_{x}>s\big)}{\widehat{\mathbb{P}}^{(e)}\big( \tau^-_{x}>t\big) } \leq C_1 \left(\frac{s}{t}\right)^{-\rho} \frac{\ell(s)}{\ell(t)},
 						\end{equation*}
 						where $C_1$ is a positive constant. On the other hand, according to Potter's Theorem in \cite{bingham1989regular} we deduce that, for any $A>1$ and $\eta>0$  there exists $t_1=t_1(A,\eta)$ such that
 						\begin{equation*}
 							\frac{\ell(s)}{\ell(t)} \leq A \max\left\{\left(\frac{s}{t}\right)^{\eta}, \left(\frac{s}{t}\right)^{-\eta}\right\}, \quad \quad t\geq s \geq t_1.
 						\end{equation*}
 						Therefore, for $t\geq s \geq t_1$
 						\begin{equation*}
 							\frac{\widehat{\mathbb{P}}^{(e)}\big( \tau^-_{x}>s\big)}{\widehat{\mathbb{P}}^{(e)}\big( \tau^-_{x}>t\big) } \leq  C_2 \left(\frac{s}{t}\right)^{-\rho}\left(\frac{s}{t}\right)^{-\eta}  =  C_2\left( \frac{t}{s}\right)^{\eta+ \rho},
 						\end{equation*}
 where $C_2$ is a positive constant.  Now plugging the later inequality back into \eqref{eq_cota_tau}, we get, as it was claimed,
 						\[
 						\begin{split}
 							\widehat{\mathbb{P}}^{(e)}\big(s< \tau^-_{x}\leq t\big) 
 							&\leq  \left(\frac{\widehat{\mathbb{P}}^{(e)}\big( \tau^-_{x}>s\big)}{\widehat{\mathbb{P}}^{(e)}\big( \tau^-_{x}>t\big) }-1 \right)  \widehat{\mathbb{P}}^{(e)}\big( \tau^-_{x}>t\big)
 							\\ & \leq  C_3\left( C_2\left( \frac{t}{s}\right)^{\eta+ \rho}-1\right)  \kappa(1/t,0) \frac{U(-x)}{\sqrt{\pi}}, \qquad \text{as}\quad s, t\to \infty,
 						\end{split}
 						\]
 						where $C_3$ is a  constant bigger than 1.
 			\end{proof}
 			
 			Recall that $\texttt{I}_{0,t}(\beta \xi)$ denotes the exponential functional of  the L\'evy process $\beta \xi$ defined in \eqref{eq_expfuncLevy}. 
 		Our  next result will be useful to control the probability of non-explosion under the event that $\{\overline{\xi}_{t}>0\}$.

 			\begin{lemma}\label{lem_critico_cota_I}
 				Let $\beta \in(-1,0)$, $C<0$ and $y>0$; and  assume that condition \eqref{eq_spitzer} holds. Then, there exists a continuous function $y\mapsto C_{\beta}(y)$ on $(0,\infty)$ such that for $t$ large enough, we have
 				\begin{equation*}
 					\widehat{\mathbb{E}}^{(e)} \left[ \exp\Big\{-y (C\beta)^{-1/\beta} {\normalfont\texttt{I}_{0,t}}(-\beta \xi)^{-1/\beta}\Big\}\right] \leq 2C_{\beta}(y) \kappa(1/t, 0).
 				\end{equation*}
				Further, 
 				\begin{equation*}
 					\lim\limits_{y\to \infty}  C_\beta(e^y) =0 \quad \quad \text{and}\quad \quad 	\lim\limits_{y\to \infty} y  C_\beta(e^y) =0.
 				\end{equation*}
 			\end{lemma}
 			\begin{proof}
 				First, we recall from  Patie and Savov \cite[ Theorem 2.20]{patie2018bernstein} that, under  Spitzer's condition \eqref{eq_spitzer}, for any continuous and bounded function  $f$  on $\mathbb{R}^+$ and any constant $a\in (0,1)$, we have
 				\begin{equation*}
 					\lim\limits_{t\to \infty} \frac{	\widehat{\mathbb{E}}^{(e)}\Big[ \texttt{I}_{0,t}(-\beta \xi)^{-a}f\big(\texttt{I}_{0,t}(-\beta \xi)\big)\Big]}{\kappa(1/t,0)} = \int_{0}^{\infty} f(x)\vartheta_a(\mathrm{d}x),
 				\end{equation*}
 				where $\vartheta_a$ is a positive measure on $(0,\infty)$. 
				
				For our purposes, we use  $f(x)=x^a \exp(-y (C\beta)^{-1/\beta}x^{-1/\beta})$ which is bounded and continuous. Thus by the latter identity we deduce that, for any $a\in (0,1)$, 
 				\begin{equation*}
 					\lim\limits_{t\to \infty} \frac{	\widehat{\mathbb{E}}^{(e)}\left[\exp\Big\{-y (C\beta)^{-1/\beta} {\normalfont\texttt{I}_{0,t}}(-\beta \xi)^{-1/\beta}\Big\}\right]}{\kappa(1/t,0)}= C_\beta(y),
 				\end{equation*}
 				where
 				\begin{equation}\label{cont_cbeta}
 					C_\beta(y):= \int_{0}^{\infty} x^a \exp\big\{-y (C\beta)^{-1/\beta}x^{-1/\beta}\big\}\vartheta_a(\mathrm{d}x),
 				\end{equation}
				which is clearly continuous.
				The latter implies that there exists $t_0>0$ such that if  $t\geq t_0$
 				\begin{equation*}
 					\widehat{\mathbb{E}}^{(e)}\left[\exp\Big\{-y (C\beta)^{-1/\beta} {\normalfont\texttt{I}_{0,t}}(-\beta \xi)^{-1/\beta}\Big\}\right] \leq 2 C_{\beta}(y) \kappa(1/t, 0),
 				\end{equation*}
				as expected. Furthermore, with the help of the Dominated Convergence Theorem, we obtain 
 				\[\begin{split}
 					\lim\limits_{y\to\infty} y C_\beta(e^y) = 2\int_{0}^{\infty} x^a\lim\limits_{y\to\infty} y \exp\big\{-e^y (C\beta)^{-1/\beta}x^{-1/\beta}\big\}\vartheta_a(\mathrm{d}x) =0
 				\end{split}\]
 				and  that $C_{\beta}(e^y)\to 0$, as $y\to \infty$. This concludes the proof.
 			\end{proof}
 			

 			The following result makes precise the statement that only paths of the L\'evy process  with a very low  running  supremum give a substantial contribution to the speed of non-explosion. 
 			
 			\begin{proposition}\label{prop_cota0}
 				Fix   $z>0$, $x<0$ and $\epsilon \in (0,1)$; and suppose  that assumptions \eqref{eq_spitzer} and \eqref{eq_Hyp} are satisfied. Then, we have for $y>x$

 				\begin{equation}\label{eq_cota2}
 					\lim\limits_{y\to \infty} \limsup_{t\to \infty}  \frac{1}{\kappa(1/t,0)} \mathbb{P}_{(z,x)}\Big(Z_t<\infty, \ \overline{\xi}_{t-\epsilon} \geq y\Big) = 0.
 				\end{equation}
 			\end{proposition}

 			\begin{proof} 
 				Fix $z>0, \ x<0$ and $\epsilon\in (0,1)$. We begin by noting that condition \eqref{eq_Hyp} allows us to find a lower bound for $v_t(0,0,\xi-x)$ in terms of the exponential functional of $\xi$. Indeed, we observe from the backward differential equation given in \eqref{eq_BDEv} that 
 	\begin{equation*}
 					\frac{\partial }{\partial s} v_t(s,\lambda e^{-x},\xi-x) \geq C v_t^{1+\beta}(s,\lambda  e^{-x},\xi-x)e^{-\beta (\xi_s-x)}, \quad v_t(t,\lambda e^{-x}, \xi-x)=\lambda e^{-x}.
 				\end{equation*}
 				Integrating between 0 and $t$, we get
 				\begin{equation*}
 					\frac{1}{v_t^\beta(0,\lambda e^{-x},\xi-x)} -\frac{1}{(\lambda e^{-x})^\beta} \leq C\beta \int_{0}^{t} e^{-\beta (\xi_s-x)}\mathrm{d} s, \quad C\beta >0.
 				\end{equation*}
 				Now, letting $\lambda\downarrow 0$ and taking into account that  $\beta\in(-1,0)$ and $C<0$, we deduce the following inequality  for all $t\geq 0$,
 				\begin{equation}\label{eq_cota_v_I}
 					v_t(0,0,\xi-x)\geq \big(C\beta \texttt{I}_{0,t}(\beta (\xi-x))  \big)^{-1/\beta},
 				\end{equation}
 				where $\texttt{I}_{0,t}(\beta (\xi-x))$ is the exponential functional of  the L\'evy process $\beta (\xi-x)$. Hence, the quenched non-explosion probability given in \eqref{eq_exploproba} may be bounded in terms of this functional. That is to say, for all $t >0$, 
 				\begin{equation*}
 					\mathbb{P}_{(z,x)}\big(Z_t<\infty \ \big|\big. \ \xi \big)  = \exp\big\{-z v_t(0,0, \xi - x)\big\}\leq \exp\Big\{-z \big(C\beta \texttt{I}_{0,t}(\beta (\xi -x))  \big)^{-1/\beta}\Big\}.
 				\end{equation*}
 				Therefore conditioning on the environment, we obtain that for any $y>x$, 
 				\begin{eqnarray}\label{eq_cotaI}
 					&&\mathbb{P}_{(z,x)}\Big(Z_t<\infty,\   \overline{\xi}_{t-\epsilon} \geq y\Big)  = \mathbb{E}_{x}^{(e)}\Big[	\mathbb{P}_{(z,x)}\big(Z_t<\infty \ \big|\big.  \ \xi \big) \mathbf{1}_{\{\overline{\xi}_{t-\epsilon} \geq y\}}\Big] \\ &  & \hspace{4.8cm}\leq  \widehat{\mathbb{E}}^{(e)}\Big[\exp\Big\{-z (C\beta)^{-1/\beta} \texttt{I}_{0,t}(-\beta \xi)^{-1/\beta}\Big\} \mathbf{1}_{\{\underline{\xi}_{t-\epsilon} \leq x-y\}}\Big] \nonumber.
 				\end{eqnarray}
 				Let  $w=x-y$ and $t_0>0$.	Now, we split the event \ $\{\tau^{-}_{w}\leq t-\epsilon\}$\   for \  $3t_0 < t$ and \ $0< \epsilon < 1$,  as follows
 				\begin{equation*}
 					\{\tau^{-}_{w}\leq t-\epsilon\}= \{0 < \tau^{-}_{w} \leq (t-t_0)/2 \} \cup \{(t-t_0)/2 < \tau^{-}_{w} \leq t-\epsilon \}.
 				\end{equation*}
 				By the monotonicity of  the mapping $t\mapsto \texttt{I}_{0,t}(-\beta \xi)$, we have that, under the event $\{0 <\tau^{-}_{w} \leq  (t-t_0)/2 \}$, the following inequalities hold
 				\[
 					0 <\tau^{-}_{w}  < \tau^{-}_{w} + \frac{t+t_0}{2}  \leq t \qquad \textrm{and} \qquad  \int_{0}^{t} e^{\beta \xi_s} \mathrm{d} s \geq  \int_{\tau^{-}_{w}}^{\tau^{-}_{w}+ \frac{t+t_0}{2}} e^{\beta \xi_s} \mathrm{d} s. 
 					\]
 				Similarly, under the event $\{(t-t_0)/2 < \tau^{-}_{w} \leq t-\epsilon\}$, we obtain
 				\[
 				\frac{t-t_0}{2} < \tau^{-}_{w} < \tau^{-}_{w} + \epsilon \leq t \qquad \textrm{and} \qquad \int_{0}^{t} e^{\beta \xi_s}  \ge \int_{\tau^{-}_{w}}^{\tau^{-}_{w}+\epsilon} e^{\beta \xi_s}.
 				\]
 				Next,  appealing to the strong Markov property  of $\xi$, we deduce 
 				\[
 				\begin{split}
 					\widehat{\mathbb{E}}^{(e)}&\left[	\exp\Big\{-z(C\beta)^{-1/\beta} 	\texttt{I}_{\tau^{-}_{w}, \tau^{-}_{w}+ \frac{t+t_0}{2}}(-\beta \xi)^{-1/\beta}\Big\};\ 0 < \tau^{-}_{w}\leq  (t-t_0)/2  \right] \\  
 					& \le   \widehat{\mathbb{E}}^{(e)} \left[\exp\left\{- ze^{-w}(C\beta)^{-1/\beta}\left(\int_{0}^{\frac{t+t_0}{2}} e^{\beta \big(\xi_{s+\tau^{-}_{w}} - \xi_{\tau^{-}_{w}}\big)}  \mathrm{d} s\right)^{-1/\beta}\right\}\mathbf{1}_{\{0 < \tau^{-}_{w}\leq  (t-t_0)/2\} }  \right]\\
 		&\le \widehat{\mathbb{E}}^{(e)} \left[\exp\Big\{-ze^{-w}(C\beta)^{-1/\beta} 	\texttt{I}_{0,\frac{t+t_0}{2}}(-\beta \xi)^{-1/\beta}\Big\}\right].
 				\end{split}
 				\]
 				Thus from  Lemma \ref{lem_critico_cota_I},  for $t$ sufficiently large, we have
 				\begin{equation*}
 					\widehat{\mathbb{E}}^{(e)} \left[\exp\Big\{-ze^{-w}(C\beta)^{-1/\beta} 	\texttt{I}_{0,\frac{t+t_0}{2}}(-\beta \xi)^{-1/\beta}\Big\}\right] \leq 2C_{\beta}(ze^{-w}) \kappa\left(\frac{2}{t+t_0}, 0\right),
 				\end{equation*}
 				where 
				$C_\beta(ze^{-w})$ are  defined as in 
				\eqref{cont_cbeta}.
				
				 				Using the same arguments as above and  Lemmas \ref{lem_cota_P} and \ref{lem_critico_cota_I}, we obtain the following sequence of inequalities for $t$ sufficiently large, 
 				\begin{equation*}
 					\begin{split}\label{eq_weakly_cota2}
 						&\widehat{\mathbb{E}}^{(e)}\Big[\exp\Big\{-z(C\beta)^{-1/\beta} 	\texttt{I}_{\tau^{-}_{w}, \tau^{-}_{w}+ \epsilon}(-\beta \xi)^{-1/\beta}\Big\}; \ (t-t_0)/2 < \tau^{-}_{w} \leq t-\epsilon  \Big] \\ &\hspace{1cm} \leq  \widehat{\mathbb{E}}^{(e)} \Big[\exp\Big\{-ze^{-w}(C\beta)^{-1/\beta} 	\texttt{I}_{0,  \epsilon}(-\beta \xi)^{-1/\beta}\Big\}\Big]\widehat{\mathbb{P}}^{(e)}\Big(  \frac{t-t_0}{2}< \tau^{-}_{w} \leq t-\epsilon   \Big)\\
						& \hspace{1cm}\le 2C_{\beta}(ze^{-w}) \kappa\left(\frac{1}{\epsilon}, 0\right)C_3\left( C_2 2^{\eta +\rho}  \left(1+ \frac{t_0 - \epsilon}{2t_0}\right)^{\eta + \rho}-1\right)\kappa\left(\frac{1}{t-\epsilon}, 0 \right) \frac{U(-w)}{\sqrt{\pi}}.
 					\end{split}
 				\end{equation*}
Hence plugging this back into \eqref{eq_cotaI} (similarly as in the proof of  Lemma 4.4 in \cite{li2018asymptotic}), we get
 	\[
 	\begin{split}
 		\limsup_{t\to \infty}	\frac{1}{\kappa(1/t,0)} &\mathbb{P}_{(z,x)}\Big(Z_t<\infty, \ \overline{\xi}_{t-\epsilon} \geq y\Big) \\ & \leq 	\limsup_{t\to \infty}	\frac{1}{\kappa(1/t,0)}	\widehat{\mathbb{E}}^{(e)}\Big[\exp\Big\{-z (C\beta)^{-1/\beta} \texttt{I}_{0,t}(-\beta \xi)^{-1/\beta}\Big\}  ; \ \tau^{-}_{w}\leq t-\epsilon\Big]   \\& \le   C_5(ze^{-w}) \limsup_{t\to \infty} 	\frac{1}{\kappa(1/t,0)} \left(\kappa\left(\frac{2}{t+t_0}, 0\right)+\kappa\left(\frac{1}{t-\epsilon}, 0\right)\right),
 		\end{split}
 	\]
 	where 
 	\begin{eqnarray*}
 						C_5(ze^{-w}):= 2 C_{\beta}(ze^{-w})\left(1  \vee \kappa\left(\frac{1}{\epsilon}, 0\right)  C_3  \left( C_2 2^{\eta +\rho}  \left(1+ \frac{t_0 - \epsilon}{2t_0}\right)^{\eta + \rho}-1\right)\frac{U(-w)}{\sqrt{\pi}}\right).
 	\end{eqnarray*}
 Note that the limsup is finite since $\theta \mapsto\kappa( \theta , 0)$ is a regular varying function at zero.  
 Finally, taking into account that the renewal function $U$ grows at most linearly, i.e. $U(y-x) = \mathcal{O}(y-x)$, recalling the asymptotic behaviour of $C_{\beta}(ze^{-w}) $ in Lemma \ref{lem_critico_cota_I} and  letting $y\to \infty$ we obtain the desired result.
 			\end{proof}
 			
			For every $z >0$ and $x<0$,  denote by  $c(z,x)$ the constant defined in Proposition \ref{prop_cotanocero}. Our next result provides some useful  properties of the mapping $x\mapsto c(z,x)U(-x)$.
 			\begin{lemma}\label{lem:propcU}
 						For each $z > 0$ the map $x \mapsto  c(z,x)U(-x)$ on $(-\infty,0)$ is decreasing, strictly positive and bounded. In particular, for each $z>0$, it holds
 						\[\lim\limits_{x\to -\infty} c(z,x)U(-x)=: B(z)\in (0,\infty).\]
 			\end{lemma}
 			
\begin{proof}
 			Note that the strictly positivity follows from the facts that the renewal function $U(-x)$ is a strictly positive function on $(-\infty,0)$ and by Lemma \ref{prop_cotanocero}. Since $\mathbb{P}_x(\overline{\xi}_t<0) \leq \mathbb{P}_y(\overline{\xi}<0)$ for $x\ge y$,  we  observe that the left-hand side on \eqref{eq_cota1} is decreasing in $x<0$,  so does  the map  $x \mapsto  c(z,x)U(-x)$.  Now, in order to see that the function is bounded from above, we first observe that similarly as in \eqref{eq_cotaI}, we have 
 				\begin{eqnarray*}
 				&&\mathbb{P}_{(z,x)}\Big(Z_t<\infty,\   \overline{\xi}_{t} < 0\Big)  = \mathbb{E}_{x}^{(e)}\Big[	\mathbb{P}_{(z,x)}\big(Z_t<\infty \ \big|\big.  \ \xi \big) \mathbf{1}_{\{\overline{\xi}_{t} < 0\}}\Big] \\ &  & \hspace{4.8cm}\leq  \widehat{\mathbb{E}}^{(e)}\Big[\exp\Big\{-z (C\beta)^{-1/\beta} \texttt{I}_{0,t}(-\beta \xi)^{-1/\beta}\Big\} \mathbf{1}_{\{\underline{\xi}_{t} \geq x\}}\Big] \\ &  & \hspace{4.8cm}\leq  \widehat{\mathbb{E}}^{(e)}\Big[\exp\Big\{-z (C\beta)^{-1/\beta} \texttt{I}_{0,t}(-\beta \xi)^{-1/\beta}\Big\}\Big]. 
 			\end{eqnarray*}
 		Hence, appealing to Lemma \ref{lem_critico_cota_I}, we obtain 
 		 	\begin{eqnarray*}
 			c(z,x)U(-x) &=& \lim\limits_{t \to \infty}\frac{1}{\kappa(1/t,0)}\mathbb{P}_{(z,x)}\Big(Z_t<\infty, \ \overline{\xi}_t <  0\Big)\\  & \leq & \lim\limits_{t \to \infty}\frac{1}{\kappa(1/t,0)}\widehat{\mathbb{E}}^{(e)}\Big[\exp\Big\{-z (C\beta)^{-1/\beta} \texttt{I}_{0,t}(-\beta \xi)^{-1/\beta}\Big\}\Big]\leq C_\beta(z),
 		\end{eqnarray*}
 where  $z\mapsto C_{\beta}(z)$ is the continuous function  defined in Lemma \ref{lem_critico_cota_I}. With this in hand, and considering that the mapping $x \mapsto  c(z,x)U(-x)$, on $(-\infty,0)$, is decreasing and strictly positive, we conclude that for every $z>0$ the limit $B(z)$ exists and it is finite and strictly positive.
 	\end{proof}
 
 			With Propositions \ref{prop_cotanocero} and \ref{prop_cota0} in hand, we may now proceed to the proof of Theorem \ref{teo_explocritica}.  The proof follows the same arguments as those used in Theorem 1.2 in \cite{bansaye2021extinction}, we provide its proof for the sake of completeness.

 			\begin{proof}[Proof of Theorem \ref{teo_explocritica}]
 				Fix $\varsigma, z >0$, $x<0$ and $\epsilon \in (0,1)$.  We begin by observing  from \eqref{CBILRE} that $\mathbb{P}_{(z,x)}(Z_t > 0)$ does not depend of the initial value $x$ of the L\'evy process $\xi$.  From Proposition \ref{prop_cota0},  we also observe that we may choose $y>0$ such that for $t$ large enough,
 				\begin{equation}
 					\mathbb{P}_{(z,x)}\Big(Z_t<\infty,\  \overline{\xi}_{t-\epsilon} \geq y\Big)   \leq \varsigma \mathbb{P}_{(z,x)}\Big(Z_t<\infty, \ \overline{\xi}_{t-\epsilon} < y\Big).
 				\end{equation} 
 				Now, note that for $t$ large enough, we get  $\{Z_{t} <\infty\}\subset \{Z_{t-\epsilon} <\infty\}$ and using the previous inequality, it follows, 
 				\begin{eqnarray*}
 					\mathbb{P}_{z}(Z_t<\infty)&=&  \mathbb{P}_{(z,x)}\Big(Z_t<\infty, \ \overline{\xi}_{t-\epsilon} \geq y\Big)  + \mathbb{P}_{(z,x)}\Big(Z_t<\infty,\  \overline{\xi}_{t-\epsilon} < y\Big) \\ &\leq & (1+\varsigma) \mathbb{P}_{(z,x-y)}\Big(Z_{t-\epsilon}<\infty, \ \overline{\xi}_{t-\epsilon} < 0\Big).
 				\end{eqnarray*}
 				In other words, for every $\varsigma>0$ there exists $y'<0$ such that  for $t$ large enough
 				\[\begin{split}
 					(1-\varsigma)\frac{\mathbb{P}_{(z,y')}\Big(Z_t<\infty, \ \overline{\xi}_{t} < 0\Big) }{\kappa(1/t,0)}&\leq \frac{	\mathbb{P}_{z}(Z_t<\infty) }{\kappa(1/t, 0)} \\ &\leq (1+\varsigma) \frac{\mathbb{P}_{(z,y')}\Big(Z_{t-\epsilon}<\infty,\  \overline{\xi}_{t-\epsilon} < 0\Big)}{\kappa(1/(t-\epsilon),0)} \frac{\kappa(1/(t-\epsilon),0)}{\kappa(1/t,0)}.
 				\end{split}\]
 				Next, using that the function $\kappa(\cdot, 0)$ is regularly varying at $0$, and then appealing to  Potter's Theorem in Bingham et al.  \cite{bingham1989regular}, we see that, for any $A>1$ and $\eta >0$,
 				\begin{eqnarray*}
 					\lim\limits_{t \to \infty} \frac{\kappa(1/(t-\epsilon),0)}{\kappa(1/t,0)}  = \lim\limits_{t \to \infty} \frac{\ell(t-\epsilon)}{\ell(t)}\left(\frac{t-\epsilon}{t}\right)^{-\rho} \leq 
 					\lim\limits_{t \to \infty}  A\left(\frac{t}{t-\epsilon}\right)^{\rho + \eta }  = A.
 				\end{eqnarray*}
 				On the other hand, according to Proposition \ref{prop_cotanocero}, there exists $0<c(z,y')<\infty$ such that 
 				\begin{equation*}
 					\lim\limits_{t \to \infty}\frac{1}{\kappa(1/t,0)}\mathbb{P}_{(z,y')}\Big(Z_t<\infty,\  \overline{\xi}_t <  0\Big) = c(z,y')U(-y').
 				\end{equation*}
 				Hence, as a consequence of the above facts, we get 
 				\begin{equation*}
 					(1-\varsigma) c(z,y')U(-y') \leq \liminf_{t \to \infty} \frac{\mathbb{P}_{z}(Z_t<\infty)}{\kappa(1/t, 0)} \leq (1+\varsigma) c(z,y')U(-y') A.
 				\end{equation*}
 				We observe that $y'$ is a sequence which may depend on $\varsigma$ and $z$. Further, this sequence $y'$ goes to $-\infty$ as $\varsigma$ goes to 0.  Thus, for any sequence $y_{\varsigma}(z)$, we have 
 				\[\begin{split}
 					0<	(1-\varsigma) c(z,y_{\varsigma}(z))U(-y_{\varsigma}(z)) &\leq \liminf_{t \to \infty} \frac{	\mathbb{P}_{z}(Z_t<\infty) }{\kappa(1/t, 0)} \\ &\leq (1+\varsigma) c(z, y_{\varsigma}(z))U(-y_{\varsigma}(z)) A < \infty,
 				\end{split}\]
 			where the strictly positivity and finiteness in the previous inequality follows from Lemma \ref{lem:propcU}. Therefore, using again  Lemma \ref{lem:propcU}, we get
 				\[\begin{split}
 					0<	\limsup_{\varsigma \to 0} (1-\varsigma) c(z,y_{\varsigma}(z))&U(-y_{\varsigma}(z)) \leq \liminf_{t \to \infty} \frac{\mathbb{P}_{z}(Z_t<\infty)}{\kappa(1/t, 0)}\\  &\leq  \liminf_{\varsigma \to 0}(1+\varsigma)  c(z,y_{\varsigma}(z))U(-y_{\varsigma}(z)) A < \infty.
 				\end{split}\] 
 				Since $A$ can be taken arbitrary close to 1,  we deduce that the inferior and superior limits are equal, finite and positive. That is to say,
 				\begin{equation*}
 					0<	 \lim\limits_{t \to \infty} \frac{\mathbb{P}_{z}(Z_t<\infty)}{\kappa(1/t, 0)}=\mathfrak{C}(z):=\lim_{\varsigma \to 0}  c(z,y_{\varsigma}(z))U(-y_{\varsigma}(z)) < \infty,
 				\end{equation*}
 				which completes the proof.
 			\end{proof}
 		\vspace{1cm}
 			
 \textbf{Acknowledgements:} Both authors would like to thank an anonymous referee and the Associated Editor for their remarks which led to a significant improvement  of this paper. 
 			
 			 			 		\vspace{0.5cm}
 			 			 		
 			\textbf{Conflict of interest:} The authors did not receive support from any organization for the submitted work.
 			 		\vspace{0.5cm}
 			
 			\textbf{Data availability}: No datasets were generated or analysed for this manuscript.

 	\bibliographystyle{abbrv}
 	\bibliography{references}

\begin{thebibliography}{10}

\bibitem{afanasyev2005criticality}
V.~I. Afanasyev, J.~Geiger, G.~Kersting, and V.~A. Vatutin.
\newblock Criticality for branching processes in random environment.
\newblock {\em Ann. Probab.}, 33(2):645--673, 2005.

\bibitem{BaDoKy}
S.~Baguley, L.~{D}\"{o}ring, and A.~E. Kyprianou.
\newblock General path integrals and stable {SDE}s.
\newblock {\em J. Eur. Math. Soc. (In press)}, 2023.

\bibitem{bansaye2021extinction}
V.~Bansaye, J.~C. Pardo, and C.~Smadi.
\newblock Extinction rate of continuous state branching processes in critical
  {L}\'{e}vy environments.
\newblock {\em ESAIM Probab. Stat.}, 25:346--375, 2021.

\bibitem{bansaye2013extinction}
V.~Bansaye, J.~C. Pardo~Millan, and C.~Smadi.
\newblock On the extinction of continuous state branching processes with
  catastrophes.
\newblock {\em Electron. J. Probab.}, 18:no. 106, 31, 2013.

\bibitem{bertoin1996levy}
J.~Bertoin.
\newblock {\em L\'{e}vy processes}, volume 121 of {\em Cambridge Tracts in
  Mathematics}.
\newblock Cambridge University Press, Cambridge, 1996.

\bibitem{bertoin1997spitzer}
J.~Bertoin and R.~Doney.
\newblock Spitzer's condition for random walks and {L}\'evy processes.
\newblock {\em Annales de l'Institut Henri Poincar\'e (B) Probability and
  Statistics}, 33(2):167--178, 1997.

\bibitem{bingham1989regular}
N.~H. Bingham, C.~M. Goldie, and J.~L. Teugels.
\newblock {\em Regular variation}, volume~27 of {\em Encyclopedia of
  Mathematics and its Applications}.
\newblock Cambridge University Press, Cambridge, 1989.

\bibitem{cardona2023speed}
N.~Cardona-Tob{\'o}n and J.~C. Pardo.
\newblock Speed of extinction for continuous state branching processes in a
  weakly subcritical {L}\'evy environment.
\newblock {\em To appear in Journal of Applied Probability}, 2024.

\bibitem{cardona2021speed}
N.~Cardona-Tob{\'o}n and J.~C. Pardo.
\newblock Speed of extinction for continuous state branching processes in
  subcritical {L}\'evy environments: the strongly and intermediate regimes.
\newblock {\em To appear in ALEA Latin American Journal of Probability and
  Mathematical Statistics}, 2024.

\bibitem{doney2007fluctuation}
R.~A. Doney.
\newblock {\em Fluctuation theory for {L}\'{e}vy processes}, volume 1897 of
  {\em Lecture Notes in Mathematics}.
\newblock Springer, Berlin, 2007.

\bibitem{maller}
K.~B. Erickson and R.~A. Maller.
\newblock Generalised {O}rnstein-{U}hlenbeck processes and the convergence of
  {L}\'{e}vy integrals.
\newblock In {\em S\'{e}minaire de {P}robabilit\'{e}s {XXXVIII}}, volume 1857
  of {\em Lecture Notes in Math.}, pages 70--94. Springer, Berlin, 2005.

\bibitem{grey1974asymptotic}
D.~R. Grey.
\newblock Asymptotic behaviour of continuous time, continuous state-space
  branching processes.
\newblock {\em J. Appl. Probab.}, 11:669--677, 1974.

\bibitem{he2018continuous}
H.~He, Z.~Li, and W.~Xu.
\newblock Continuous-state branching processes in {L}\'{e}vy random
  environments.
\newblock {\em J. Theoret. Probab.}, 31(4):1952--1974, 2018.

\bibitem{ikeda2014stochastic}
N.~Ikeda and S.~Watanabe.
\newblock {\em Stochastic differential equations and diffusion processes}.
\newblock Elsevier, 2014.

\bibitem{KolSav}
M.~Kolb and M.~Savov.
\newblock A characterization of the finiteness of perpetual integrals of
  {L}\'{e}vy processes.
\newblock {\em Bernoulli}, 26(2):1453--1472, 2020.

\bibitem{kwasnicki2013suprema}
M.~Kwa\'{s}nicki, J.~Ma{\l}ecki, and M.~Ryznar.
\newblock Suprema of {L}\'{e}vy processes.
\newblock {\em Ann. Probab.}, 41(3B):2047--2065, 2013.

\bibitem{kyprianou2014fluctuations}
A.~E. Kyprianou.
\newblock {\em Fluctuations of {L}\'{e}vy processes with applications}.
\newblock Universitext. Springer, Heidelberg, second edition, 2014.

\bibitem{li2018asymptotic}
Z.~Li and W.~Xu.
\newblock Asymptotic results for exponential functionals of {L}\'{e}vy
  processes.
\newblock {\em Stochastic Process. Appl.}, 128(1):108--131, 2018.

\bibitem{palau2018branching}
S.~Palau and J.~C. Pardo.
\newblock Branching processes in a {L}\'{e}vy random environment.
\newblock {\em Acta Appl. Math.}, 153:55--79, 2018.

\bibitem{palau2016asymptotic}
S.~Palau, J.~C. Pardo, and C.~Smadi.
\newblock Asymptotic behaviour of exponential functionals of {L}\'{e}vy
  processes with applications to random processes in random environment.
\newblock {\em ALEA Lat. Am. J. Probab. Math. Stat.}, 13(2):1235--1258, 2016.

\bibitem{patie2018bernstein}
P.~Patie and M.~Savov.
\newblock Bernstein-gamma functions and exponential functionals of {L}\'{e}vy
  processes.
\newblock {\em Electron. J. Probab.}, 23:Paper No. 75, 101, 2018.

\bibitem{persson1975generalization}
J.~Persson.
\newblock A generalization of {C}arath{\'e}odory's existence theorem for
  ordinary differential equations.
\newblock {\em Journal of Mathematical Analysis and Applications},
  49(2):496--503, 1975.

\bibitem{ScSVon}
R.~L. Schilling, R.~Song, and Z.~Vondra\v{c}ek.
\newblock {\em Bernstein functions}, volume~37 of {\em De Gruyter Studies in
  Mathematics}.
\newblock Walter de Gruyter \& Co., Berlin, second edition, 2012.
\newblock Theory and applications.

\end{thebibliography}

 \end{document}